\documentclass[12pt]{article}
\usepackage[latin1]{inputenc}

\usepackage{amsfonts}
\usepackage{amssymb}
\usepackage{amsthm}
\usepackage{amsmath}
\usepackage{graphicx}
\usepackage{color}

\usepackage{amscd}
\usepackage{multirow}

\newcommand{\CC}{{\mathbb C}}
\newcommand{\C}{{\mathbb C}}
\newcommand{\cD}{{\mathcal D}}

\newcommand{\cH}{{\mathcal H}}
\newcommand{\cO}{{\mathcal O}}
\newcommand{\cQ}{{\mathcal Q}}
\newcommand{\cM}{{\mathcal M}}

\newcommand{\R}{{\mathbb R}}
\newcommand{\ZZ}{{\mathbb Z}}
\newcommand{\Z}{{\mathbb Z}}
\newcommand{\NN}{{\mathbb N}}
\newcommand{\QQ}{{\mathcal Q}}

\newcommand\N{{\mathbb{N}}}

\newcommand\HAb {{H_A (\beta )}}

\newcommand\D{{\mathcal{D}}}
\newcommand\Q{{\mathbb{Q}}}

\newcommand\M{{\mathcal{M}}}

\newtheorem{teo}{Theorem}[section]
\newtheorem{prop}[teo]{Proposition}
\newtheorem{defi}[teo]{Definition}
\newtheorem{ejem}[teo]{Example}
\newtheorem{lema}[teo]{Lemma}
\newtheorem{cor}[teo]{Corollary}
\newtheorem{nota}[teo]{Remark}

\title{Irregular hypergeometric $\D$-modules}

\author{María-Cruz Fernández-Fernández \thanks{Supported by the F.P.U. Fellowship AP2005-2360, Spanish Ministry of Education, and partially
supported by MTM2007-64509 and FQM333.
e.mail address: {\tt mcferfer@us.es}}\\
Departamento de \'{A}lgebra \\ Universidad de Sevilla}

\date{July 15, 2009}

\begin{document}

\maketitle

\begin{abstract}
We study the irregularity of hypergeometric $\D$-modules
$\mathcal{M}_A (\beta)$ via the explicit construction of Gevrey
series solutions along coordinate subspaces in $X=\C^n$. As a
consequence, we prove that along coordinate hyperplanes the
combinatorial characterization of the slopes of $\M_A (\beta)$ given
by M. Schulze and U. Walther in \cite{SW} still holds for any full
rank integer matrix $A$. We also provide a lower bound for the
dimensions of the spaces of Gevrey solutions along coordinate
subspaces in terms of volumes of polytopes and prove the equality
for very generic parameters. Holomorphic solutions outside the
singular locus of $\M_A (\beta )$ can be understood as Gevrey
solutions of order one along $X$ at generic points and so they are
included as a particular case.
\end{abstract}

\tableofcontents

\section{Introduction}

This paper is devoted to the study of the irregularity of the
GKZ-hypergeometric $\D$-modules. To this end we explicitly construct
Gevrey series solutions along coordinate subspaces in $\C^n$. Let us
first recall some notions and results about the irregularity in
$\D$-module Theory.

\vspace{.3cm}

\indent Let $X$ be a complex manifold and $\D_X$ the sheaf of linear
partial differential operators with coefficients in the sheaf of
holomorphic functions $\cO_X$.

\vspace{.3cm}

One fundamental problem in the study of the irregularity of a
holonomic $\D_X$-module $\mathcal{M}$ is the description of its {\em
analytic slopes} along smooth hypersurfaces $Y$ in $X$ (see Z.
Mebkhout \cite{Mebkhout}). An analytic slope is a gap $s > 1$ in the
Gevrey filtration $\operatorname{Irr}_Y^{(s)}(\M )$ of the
irregularity complex $\operatorname{Irr}_Y (\mathcal{M})$ (see
Definitions \ref{irrefularity-s} and \ref{trans-slope}).

\vspace{.3cm}

Y. Laurent also defined the {\em algebraic slopes} of $\mathcal{M}$
along a smooth variety $Z$ (\cite{Laurent-ast-85},
\cite{Laurent-ens-87}) as those real numbers $s>1$ such that the
$s$-micro-characteristic variety of $\mathcal{M}$ with respect to
$Z$ is not homogeneous with respect to the filtration by the order
of the differential operators. He proved that the set of slopes of
$\M$ along $Z$ is a finite set of rational numbers (see
\cite{Laurent-ens-87}).

\vspace{.3cm}

When $\M$ is a holonomic $\D$-module and $Y$ is a smooth
hypersurface, the Comparison Theorem of the slopes (due to Laurent
and Mebkhout \cite{Laurent-Mebkhout}) states that the algebraic
slopes coincide with the analytic ones. However, as far as we know,
the analytic slopes of a holonomic $\D$-module along varieties $Z$
of codimension greater than one are not defined yet in the
literature. One problem is that the complexes
$\operatorname{Irr}_Z^{(s)} (\mathcal{M})$ and $\operatorname{Irr}_Z
(\mathcal{M})$ are constructible but they are not necessarily
perverse in such a case (see examples in \cite{Mebkhout4}). The
category of perverse sheaves is an abelian category while the
category of cons\-tructible sheaves is just additive (see
\cite{BBD}).

\vspace{.3cm}

The description of the Gevrey series solutions of a holonomic
$\D$-module $\M$ along a smooth variety $Z$ is another fundamental
problem in the study of its irregularity. If $Y$ is a smooth
hypersurface the index of any non convergent Gevrey solution of $\M$
along $Y$ is an analytic slope of $\M$ along $Y$ (see Definition
\ref{trans-slope}).

\vspace{.3cm}

From now on we consider the complex manifold $X=\C^n$ and denote
$\D:=\D_X$. We also will write $\partial_i:=\frac{\partial}{\partial
x_i}$ for the $i$-th partial derivative.

\vspace{.3cm}

Hypergeometric systems were introduced by Gel'fand, Graev, Kapranov
and Zelevinsky (see \cite{GGZ} and \cite{GKZ}) and they are
associated with a pair $(A,\beta)$ where $A$ is a full rank $d\times
n$ matrix $A=(a_{ij})$ with integer entries ($d\leq n$) and $\beta
\in \C^d$ is a vector of complex parameters. They are left ideals
$H_A (\beta )$ of the Weyl algebra $\C [x_1 ,\ldots ,x_n ]\langle
\partial_1 ,\ldots ,\partial_n \rangle$ generated by the following
set of differential operators:

\begin{equation}
\Box_u := \partial^{u_+}-\partial^{u_{-}} \; \;\; \; \mbox{ for }
u\in \ZZ^n , \; Au=0 \label{Toric-operators}
\end{equation} where $u=u_+-u_-$ and  $u_+, u_- \in \NN^n$ have disjoint
supports, and

\begin{equation}
E_i - \beta_i := \sum_{j=1} ^n a_{ij}x_j\partial_j -\beta_i \; \; \;
\; \mbox{ for } i=1,\ldots, d \label{Euler-operators}
\end{equation}

\vspace{.3cm}

The hypergeometric $\cD$--module associated with the pair $(A,\beta
)$ is the quotient sheaf $\cM_A(\beta)=\cD /\cD H_A (\beta )$.

\vspace{.3cm}

The operators given in (\ref{Toric-operators}) are called the toric
operators associated with $A$ and they generate the so-called toric
ideal $I_A \subseteq \CC[\partial_1,\ldots,\partial_n] $ associated
with $A$. It is a prime ideal whose zeros variety $\mathcal{V}(I_A
)\subseteq \C^n$ is an affine toric variety with Krull dimension $d$
(see for example \cite{Sturm}). The operator $E_i$ is called the
$i$-th Euler operator associated with $A$ for $i=1,\ldots ,d$.

\vspace{.3cm}

A good introduction for the theory of hypergeometric systems is
\cite{SST}. These systems are known to be holonomic and their
holonomic rank (equivalently, the dimension of the space of
holomorphic solutions at nonsingular points) is the normalized
volume of the matrix $A=(a_i )_{i=1}^n \in \Z^{d\times n}$ with
respect to the lattice $\Z A :=\sum_{i=1}^n \Z a_i \subseteq \Z^d$
(see Definition \ref{def-normal-vol}) when either $\beta$ is generic
or $I_A$ is Cohen-Macaulay (see \cite{GKZ}, \cite{Adolphson}). For
results about rank-jumping parameters $\beta$ see \cite{MMW},
\cite{Berkesch} and the references therein. Several authors have
studied the holomorphic solutions at nonsingular points of
$\cM_A(\beta)$ (see \cite{GKZ}, \cite{SST} and
\cite{Ohara-Takayama}).

\vspace{.3cm}

A theorem of R. Hotta \cite[Ch. II, §6.2, Thm.]{Hotta} assures that
when the toric ideal $I_A$ is homogeneous the hypergeometric
$\D$-module $\cM_A (\beta )$ is regular holonomic. The converse to
this theorem was proved by Saito, Sturmfels and Takayama \cite[Thm.
2.4.11]{SST} when $\beta$ is generic and by Schulze and Walther
\cite[Corollary 3.16]{SW} when $A$ is a \emph{pointed} matrix such
that $\Z A=\Z^d $. A matrix $A$ is said to be pointed if its columns
$a_1 ,\dots ,a_n$ lie in a single open linear half-space of $\R^d$
(equivalently, the associated affine toric variety
$\mathcal{V}(I_A)$ passes through the origin). On the other hand,
when $A$ is non-pointed then $\M_A (\beta )$ is never regular
holonomic: the existence of a toric operator $\partial^u -1 \in
I_A$, $u\in \N^n$, implies that the holonomic rank of some initial
ideals of $H_A (\beta )$ is zero and this cannot hold for regular
holonomic ideals with positive rank (see \cite[Thm. 2.5.1.]{SST}).

\vspace{.3cm}

Let us explain the structure of this paper. In Section
\ref{section-general} we recall some general definitions (Gevrey
series, irregularity and analytic slopes of a holonomic $\D$-module)
and prove Lemma \ref{lema-slope} which concerns the slopes of
holonomic $\D$-modules and will be used in the sequel.

\vspace{.3cm}

In Section \ref{construction-Gevrey} we consider a simplex $\sigma$,
i.e., a set $\sigma\subseteq\{1,\ldots ,n\}$ such that
$A_{\sigma}=(a_i )_{i\in \sigma}$ is an invertible submatrix of $A$,
and we use the $\Gamma$--series introduced in \cite{GKZ} and
slightly generalized in \cite{SST} to explicitly construct a set of
linearly independent Gevrey solutions of $\M_A (\beta )$ along
$Y_{\sigma}=\{x_i =0 :\; i\notin \sigma\}$. The cardinal of this set
of solutions is the normalized volume of $A_{\sigma }$ with respect
to the lattice $\Z A$ and we prove that they are Gevrey series of
order $s=\max \{|A_{\sigma}^{-1}a_i | :\; i\notin \sigma \}$ along
the coordinate subspace $Y=\{x_i =0 :\; |A_{\sigma}^{-1}a_i|>1
\}\supseteq Y_{\sigma }$.  Moreover, we also prove that $s$ is their
Gevrey index when $\beta$ is very generic.

\vspace{.3cm}

In Section \ref{slopes-simplex-section} we construct for any simplex
$\sigma$ and for all $\beta$ a set of Gevrey series along $Y$ with
index $s$ that are solutions of $\M_A (\beta )$ modulo the sheaf of
Gevrey series with lower index. This implies for $s>1$ that $s$ is a
slope of $\M_A (\beta)$ along $Y$ when $Y$ is a hyperplane by Lemma
\ref{lema-slope}.

\vspace{.3cm}

In Section \ref{slopes-hyperplanes} we describe all the slopes of
$\M_A (\beta )$ along coordinate hyperplanes $Y$ at any point $p\in
Y$ (see Theorem \ref{characterization-slopes}). To this end, and
using some ideas of \cite{SW}, we prove that the
$s$-micro-characteristic varieties with respect to $Y$ of $\M_A
(\beta )$ are homogeneous with respect to the order filtration for
all $s\geq 1$ but a finite set of candidates $s$ to be algebraic
slopes. Then we use the results in Sections
\ref{construction-Gevrey} and \ref{slopes-simplex-section} to prove
that all the candidates $s$ to be algebraic slopes along hyperplanes
occur as the Gevrey index of a Gevrey series solution of $\M_A
(\beta )$ modulo convergent series and thus they are analytic
slopes. In particular we prove that the set of algebraic slopes of
$\M_A (\beta )$ along any coordinate hyperplane is contained in the
set of analytic slopes without using the Comparison Theorem of the
slopes (due to Laurent and Mebkhout \cite{Laurent-Mebkhout}). We use
this theorem in the converse direction to prove that there are no
more slopes. M. Schulze and U. Walther \cite{SW} described in a
combinatorial way all the algebraic slopes of $\M_A (\beta )$ along
coordinate subspaces assuming that $\Z A=\Z^d $ and that $A$ is
pointed. Previous computations in the particular cases $d=1$ and
$n=d+1$ of the slopes along coordinate hyperplanes appear in
\cite{Castro-Takayama}, \cite{hartillo_trans} and
\cite{hartillo_rmi}.

\vspace{.3cm}

In Section \ref{triangulations-umbrella} we recall the definition of
regular triangulation of a matrix and make some remarks that will be
used in Section \ref{basis-Gevrey}.

\vspace{.3cm}

In Section \ref{basis-Gevrey} we use the Gevrey series constructed
in Section \ref{construction-Gevrey} and convenient regular
triangulations of the matrix $A$ to provide a lower bound for the
dimensions of the Gevrey solution spaces. In particular, the lower
bound that we obtain for the dimension of the formal solution space
of $\M_A (\beta )$ along any coordinate subspace $Y_{\tau}=\{
x_{i}=0:\; i\notin \tau \}$ at generic points of $Y_{\tau}$ is
nothing but the normalized volume of the matrix $A_{\tau}$ with
respect to $\Z A$.

\vspace{.3cm}

In Section \ref{basis-Gevrey2} we prove that this lower bound is
actually an equality for very generic parameters $\beta \in \C^d$
and then we have the explicit description of the basis of the
corresponding Gevrey solution space. Example \ref{contraejemplo}
shows that this condition on the parameters is necessary in general
to obtain a basis. This example also points out a special
phenomenon: some algebraic slopes of $\M_A (\beta )$ along
coordinate subspaces of codimension greater than one do not appear
as the Gevrey index of any formal solution modulo convergent series.

\vspace{.3cm}

Finally, in Section \ref{section-pointed} we assume some conditions
($\Z A =\Z^d$, $A$ is pointed, $\beta$ is non-rank-jumping and $Y$
is a coordinate hyperplane) in order to use some multiplicity
formulas for the $s$-characteristic cycles of $\M_{A}(\beta )$
obtained by M. Schulze and U. Walther in \cite{SW} and general
results on the irregularity of holonomic $\D$-modules due to Y.
Laurent and Z. Mebkhout \cite{Laurent-Mebkhout} to compute the
dimension of $\mathcal{H}^0 (\operatorname{Irr}^{(s)}_Y (\M_A (\beta
)))_p$ for generic points $p\in Y$. Then the set of the classes in
$\cQ_Y (s)$ of the Gevrey solutions that we construct along a
hyperplane is a basis for very generic parameters. Moreover, since
$\operatorname{Irr}_Y^{(s)} (\M_A (\beta ))$ is a perverse sheaf on
$Y$ by a theorem of Z. Mebkhout \cite{Mebkhout}, we know that for
all $i\geq 1$ the $i$-th cohomology sheaf of
$\operatorname{Irr}_Y^{(s)} (\M_A (\beta ))$ has support contained
in a subvariety of $Y$ with codimension $i$. This gives the stalk of
the cohomology of $\operatorname{Irr}_Y^{(s)}(\M_A (\beta ))$ at
generic points of $Y$.

\vspace{.3cm}

This paper is very related with \cite{FC1} and \cite{FC2}. In
\cite{FC2} we use deep results in $\D$-module Theory and restriction
theorems to reduce the computation of the cohomology sheaves of
$\operatorname{Irr}_Y^{(s)} (\M_A (\beta ))$ for a pointed one-row
matrix $A$ to the case associated with a $1\times 2$ matrix (that we
solved by elementary methods in \cite{FC1}). We also described a
basis of the Gevrey solutions in both articles. However, the problem
of the combinatorial description of the higher cohomology of the
irregularity sheaves $\operatorname {Irr}_Y^{(s)}(\M_A (\beta ))$ at
non generic points of $Y$ for general hypergeometric $\D$-modules
seems much more involved since free resolutions are very difficult
to compute.

\begin{center}
\textbf{Acknowledgments}
\end{center}

I am grateful to my advisor Francisco-Jesús Castro-Jiménez for
introducing me to this topic, many helpful conversations and useful
suggestions. I also thank Alicia Dickenstein and Federico N.
Martínez for interesting comments, including Remark \ref{nota-DM}.

\section{Gevrey series and slopes of $\D$-modules}\label{section-general}

Let $Y\subseteq X=\C^n$ be a smooth analytic subvariety and
$\mathcal{I}_Y \subseteq \cO_X $ its defining ideal. The formal
completion of $\cO_X$ along $Y$ is given by

$$\cO_{\widehat{X|Y}}:=\lim_{\stackrel{\longleftarrow}{k}} \cO_X /\mathcal{I}_Y^k .$$

In this section, we can assume that locally $Y= Y_{\tau} =\{x_i =0
:\; i\notin \tau \}$ for $\tau \subseteq \{1,\ldots , n\}$ with
cardinal $r=\dim_{\C } (Y)$. We will denote $x_{\mathcal{\tau}}:=(
x_i : i \in \mathcal{\tau} )$ and $\overline{\tau}=\{ 1, \ldots
,n\}\setminus \tau$. A germ of $\cO_{\widehat{X|Y}}$ at $p \in Y$
has the form

$$f=\sum_{\alpha\in \N^{n-r}}
f_{\alpha}(x_{\tau} ) x_{\overline{\tau}}^{\alpha} \in
\cO_{\widehat{X|Y},p}\subseteq \C \{ x_{\tau} - p_{\tau}\}[[
x_{\overline{\tau}} ]]$$ where $f_{\alpha}(x_{\tau} ) \in
\cO_{Y}(U)$ for certain nonempty relatively open subset $U \subseteq
Y$, $p\in U$. The germs of $\cO_{\widehat{X|Y}}$ are called formal
series along $Y$.

\begin{defi} A formal series $$f=\sum_{\alpha\in \N^{n-r}}
f_{\alpha}(x_{\tau} ) x_{\overline{\tau}}^{\alpha} \in \C \{
x_{\tau} - p_{\tau}\}[[ x_{\overline{\tau}} ]]$$ is said to be
Gevrey of multi-order $\mathbf{s}=(s_i)_{i \notin \tau} \in
\R^{n-r}$ along $Y$ at $p\in Y$ if the series
$$\rho_{\mathbf{s}}^{\tau } (f):= \sum_{\alpha\in \N^{n-r}}
\frac{f_{\alpha}(x_{\tau} )}{\alpha!^{\mathbf{s}-\mathbf{1}}}
x_{\overline{\tau}}^{\alpha}$$ is convergent at $p$. Here we denote
$\alpha!^{\mathbf{s}-\mathbf{1}}=\prod_{i\notin \tau} (\alpha_i
!)^{s_i -1}$.
\end{defi}

\begin{defi}\label{definition-Gevrey}
A formal series $$f=\sum_{\alpha\in \N^{n-r}} f_{\alpha}(x_{\tau} )
x_{\overline{\tau}}^{\alpha} \in \C \{ x_{\tau} - p_{\tau}\}[[
x_{\overline{\tau}} ]]$$ is said to be Gevrey of order $s\in \R$
along $Y$ at $p\in Y$ if the series $$\rho_{s}^{\tau } (f):=
\sum_{\alpha\in \N^{n-r}} \frac{f_{\alpha}(x_{\tau}
)}{(\alpha!)^{s-1}} x_{\overline{\tau}}^{\alpha}$$ is convergent at
$p$.

\vspace{.3cm}

Moreover, if $\rho_{s'}^{\tau} (f)$ is not convergent at $p$ for any
$s' <s$ then $s$ is said to be the Gevrey index of $f$ along $Y$ at
$p$. It is clear that such a series $f$ belongs to
$\cO_{\widehat{X|Y},p}$ and we denote by $\cO_{X|Y}(s)$ the subsheaf
of $\cO_{\widehat{X|Y}}$ whose germs are Gevrey series of order $s$
along $Y$.
\end{defi}

\begin{nota}
Notice that any Gevrey series of multi-order $\mathbf{s}=(s_i
)_{i\notin \tau}$ along $Y$ at $p\in Y$ is also a Gevrey series of
order $s=\max \{ s_i :\; i\notin \tau\}$ along $Y$ at $p$.
\end{nota}

For $s=1$ we have that $\cO_{X|Y}(1)=\cO_{X|Y}$ is the restriction
of $\cO_X$ to $Y$ and by convention $\cO_{X|Y}(+\infty
)=\cO_{\widehat{X|Y}}$.

\vspace{.3cm}

We denote by $\cQ_Y$ the quotient sheaf
$\cO_{\widehat{X|Y}}/\cO_{X|Y}$ and by $\cQ_Y(s)$ its subsheaf
$\cO_{X|Y}(s)/\cO_{X|Y}$ for $1\leq s\leq \infty$.

\begin{defi}{\rm \cite[Definition 6.3.1]{Mebkhout}}\label{irrefularity-s} For each $1\leq s \leq \infty$,
the irregularity complex of order $s$ of $\M$ along $Y$ is
$$\operatorname{Irr}_{Y}^{(s)} (\M ):=\R \cH om_{\cD_X }(\M , \cQ_Y
(s)).$$ The irregularity complex of $\M$ along $Y$ is
$\operatorname{Irr}_Y (\M):= \operatorname{Irr}_{Y}^{(\infty)} (\M
)$.
\end{defi}

Z. Mebkhout proved in \cite[Th. 6.3.3]{Mebkhout} that for any
holonomic $\cD_X$--module $\M$ and any smooth hypersurface $Y\subset
X$ the complex $\operatorname{Irr}_Y^{(s)}(\M)$ is a perverse sheaf
on $Y$ for $1\leq s\leq \infty$. Furthermore, the sheaves
$\operatorname{Irr}_{Y}^{(s)} (\M )$, $s \geq 1$, determine an
increasing filtration of $\operatorname{Irr}_{Y} (\M )$. This
filtration is called the Gevrey filtration of $\operatorname{Irr}_Y
(\M)$ (see \cite[Sec. 6]{Mebkhout}).

\vspace{.3cm}

Assume for the remainder of this section that $Y$ is a smooth
hypersurface.

\begin{defi} {\rm \cite[Sec. 2.4]{Laurent-Mebkhout}}\label{trans-slope} A number $s > 1$ is said to be
an analytic slope of $\M$ along $Y$ at a point $p\in Y$ if $p$
belongs to the analytic closure of the set:

$$\{q\in Y: \; \operatorname{Irr}_{Y}^{(s')}
(\M )_q \neq \operatorname{Irr}_{Y}^{(s)} (\M )_q , \; \forall s'<s
\} .$$
\end{defi}

Let us denote $\cO_{X|Y}(<s):=\cup_{s'<s} \cO_{X|Y}(s')$ for $s\in
\R$. For the sake of completeness we include a proof of the
following result.

\begin{lema}\label{lema-slope}
Let $\mathcal{M}$ be a holonomic $\D$-module such that there exists
a series $f\in \cO_{X|Y}(s)_p $ with Gevrey index $s>1$ whose class
in $$ (\cO_{X|Y}(s) /\cO_{X|Y}(<s))_p $$ is a solution of $\M$, for
all $ p$ in a relatively open set $U \subseteq Y$. Then $s$ is a
slope of $\mathcal{M}$ along $Y$ at any point in the closure of $U$.
\end{lema}

\begin{proof}
Any holonomic $\D$-module is cyclic (see \cite[Proposition
3.1.5]{Bjork}). Thus, we can assume without loss of generality that
$\M=\D /\mathcal{I}$ with $\mathcal{I}$ a sheaf of ideals generated
by some differential operators $P_1 ,\ldots ,P_m \in \D (U)$. Then,
by the assumption, there exists $s_i <s$ such that $P_i (f_p )\in
\cO_{X|Y}(s_i)_p$, $i=1,\ldots ,m$. For $s'=\max \{ s_i \} <s$ we
have that $(\overline{P_i (f)})_{i=1}^m \in (\QQ_Y (s') )^m$
verifies all the left $\D$-relations verified by $(P_i )_{i=1}^m  $.
Thus, we can consider its class in $\mathcal{H}^1
(\operatorname{Irr}_Y^{(s')} (\mathcal{M}))$.

\vspace{.3cm}

Since $Y$ is a smooth hypersurface, $\operatorname{Irr}_Y^{(s')}
(M)$ is a perverse sheaf on $Y$ \cite{Mebkhout}. In particular, the
support $S$ of the sheaf $\mathcal{H}^1 (\operatorname{Irr}_Y^{(s')}
(\mathcal{M}))$ has at most dimension equal to $\dim Y - 1$, so the
relatively open set $U'=U\setminus \overline{S} \subseteq Y$
verifies that $\mathcal{H}^1 (\operatorname{Irr}_Y^{(s')}
(\mathcal{M}))_{|U'} =0$ and its closure is equal to the one of $U$.

\vspace{.3cm}

In particular the class of $(\overline{P_i (f)})_{i=1}^m \in (\QQ_Y
(s') )^m$ in $\mathcal{H}^1 (\operatorname{Irr}_Y^{(s')}
(\mathcal{M}))_{|U'} $ is zero. This implies the existence of
$\overline{h} \in \QQ_Y (s')$ such that $(\overline{P_i
(h)})_{i=1}^m = (\overline{P_i (f)})_{i=1}^m$ in $(Q_Y (s')_p )^m$.
Equivalently, $P_i (f-h)$ is convergent at any point of $U'$ for all
$i=1,\ldots ,m$, and we also have that $f-h$ has Gevrey index $s$
because $f$ has Gevrey index $s$ and $h$ has Gevrey index $s'<s$.

\vspace{.3cm}

The last assertion means that $$\overline{f-h} \in
\mathcal{H}om_{\D}(\mathcal{M},\QQ_{Y}(s))_{|U'}\setminus
\bigcup_{s'<s} \mathcal{H}om_{\D}(\mathcal{M},\QQ_{Y}(s' ))_{|U'} $$
and therefore $s$ is a slope of $\mathcal{M}$ along $Y$ at any point
in the closure of $U$.
\end{proof}

\begin{nota}\label{nota-slope} For all $s\geq 1$,
$\operatorname{Irr}_Y^{(s)} (\M)$ is a constructible sheaf (see
\cite{Mebkhout}). Then there exists a Whitney stratification $\{
Y_{\alpha}(s) \}_{\alpha }$ of $Y$ such that
$\mathcal{H}^{i}(\operatorname{Irr}_Y^{(s)} (\M))_{| Y_{\alpha}(s)}$
are locally constant sheaves. If $Y$ is an irreducible algebraic
hypersurface and $Y_{\alpha}(s)$ are algebraic subvarieties then the
set $Y_{\gamma}(s)=Y\setminus \cup_{\dim Y_{\alpha }(s)<n-1}
Y_{\alpha}(s)$ is a connected stratum (see \cite[Théorème
2.1.]{Grothendieck}). Since $U\cap Y_{\gamma}(s)$ is a relatively
open set in $Y_{\gamma}(s)$ and $s$ is a slope of $\M$ along $Y$ at
any point of $U$, we have that $s$ is a slope of $\mathcal{M}$ along
$Y$ at any point of $Y_{\gamma}(s)$. This implies that $s$ is a
slope of $\mathcal{M}$ along $Y$ at any point of $Y$ by Definition
\ref{trans-slope} because $Y$ is the analytic closure of
$Y_{\gamma}(s)$.
\end{nota}

\section{Gevrey solutions of $\M_A (\beta )$ associated with a simplex}\label{construction-Gevrey}

Let $A=(a_1  \cdots a_n )$ be a full rank matrix with columns $a_j
\in \Z^d$ and $\beta \in \C^d$.

\vspace{.3cm}

For any set $\tau \subseteq \{1,\ldots ,n\}$ let
$\operatorname{conv}(\tau)$ be the convex hull of $\{ a_i : \; i\in
\tau \}\subseteq \R^d$ and let $\Delta_{\tau}$ be the convex hull of
$\{ a_i : \; i\in \tau \}\cup \{\mathbf{0} \}\subseteq \R^d$. We
shall identify $\tau$ with the set $\{a_i : \; i \in \tau \}$ and
with $\operatorname{conv}(\tau )$. We also denote by $A_{\tau }$ the
matrix given by the columns of $A$ indexed by $\tau$.

\vspace{.3cm}

We fix a set $\sigma\subseteq \{1,\ldots ,n\}$ with cardinal $d$ and
$\det (A_{\sigma})\neq 0$ throughout this section. Then
$\Delta_{\sigma }$ is a $d$-simplex and $\sigma$ is a
$(d-1)$-simplex. The normalized volume of $\Delta_{\sigma}$ with
respect to $\Z A$ is $$\operatorname{vol}_{\Z A}(\Delta_{\sigma}
)=\dfrac{d! \operatorname{vol}(\Delta_{\sigma})}{[\Z^d :\Z
A]}=\dfrac{|\det (A_{\sigma})|}{[\Z^d :\Z A]}$$ where
$\operatorname{vol}(\Delta_{\sigma})$ denotes the Euclidean volume
of $\Delta_{\sigma}$. The aims of this section are: 1) to explicitly
construct $\operatorname{vol}_{\Z A}(\Delta_{\sigma} )$ linearly
independent formal solutions of $\M_A (\beta )$ along the subspace
$Y_{\sigma}=\{x_i=0 : \; i\notin \sigma \}$ at any point of
$Y_{\sigma}\cap \{x_j \neq 0: \; j \in \sigma \}$ and 2) to prove
that these series are Gevrey series along $Y_{\sigma}$ of
multi-order $(s_i )_{i\notin \sigma} $ with $s_i
=|A_{\sigma}^{-1}a_i |$.

\vspace{.3cm}

We reorder the variables in order to have $\sigma=\{1,\dots ,d\}$
for simplicity. Then a basis of $\ker (A) =\{u\in \mathbb{Q}^n : \;
Au=0 \}$ is given by the columns of the matrix:

$$B_{\sigma}=\left( \begin{array}{c}
                      -A_{\sigma}^{-1}A_{\overline{\sigma}} \\
                      I_{n-d}
                    \end{array}
 \right)=\left(\begin{array}{cccc}
               -A_{\sigma}^{-1} a_{d+1} & -A_{\sigma}^{-1}a_{d+2} & \cdots & -A_{\sigma}^{-1}a_{n} \\
               1 & 0 &  & 0 \\
               0 & 1 &  & 0 \\
          \vdots &   & \ddots & \vdots \\
               0 & 0 &   & 1
             \end{array}\right)$$ For $v\in \C^n$ with $A v =\beta$ the $\Gamma$--series defined in
\cite{GKZ}:

$$\varphi_v : =\sum_{u\in L_A} \frac{1}{\Gamma (v+u+1)} x^{v+u}$$ is
formally annihilated by the differential operators
(\ref{Toric-operators}) and (\ref{Euler-operators}). Here $\Gamma$
is the Euler Gamma function and $L_A :=\ker (A)\cap \Z^n$. Notice
that $\varphi_v$ is zero if and only if $(v + L_A)\cap (\C \setminus
\Z_{<0})^{n}=\emptyset$. In contrast, the series $\varphi_v$ does
not define a formal power series at any point if $v\in \C^n$ is very
generic.

\vspace{.3cm}

Observe that $v:=(A_{\sigma}^{-1}\beta , \mathbf{0}) \in \C^n$
satisfies $Av=\beta$ and so the vectors
$$v^{\mathbf{k}}=(A_{\sigma}^{-1}(\beta -
\sum_{i\notin \sigma} k_i a_i),\mathbf{k})$$ where $ \mathbf{k}=(k_i
)_{i\notin \sigma }\in \N^{n-d} $. Hence, according to Lemma 1 in
Section 1.1 of \cite{GKZ}, we have that the formal series along
$Y_{\sigma}:=\{x_i =0 :\; i\notin \sigma \}$ at any point of
$Y_{\sigma}\cap \{x_j \neq 0:\; j\in \sigma \} $:

$$\varphi_{v^{\mathbf{k}}}=x_{\sigma}^{A_{\sigma}^{-1}\beta}\sum_{\mathbf{k}+
\mathbf{m}\in \Lambda_{\mathbf{k}}} \frac{x_{\sigma}^{-
A_{\sigma}^{-1}(\sum_{i\notin \sigma} (k_i +m_i ) a_i)}
x_{\overline{\sigma}}^{\mathbf{k} +\mathbf{m}}}{\Gamma
(A_{\sigma}^{-1}(\beta -\sum_{i\notin \sigma} (k_i +m_i )
a_i)+\mathbf{1})(\mathbf{k}+\mathbf{m})!}
$$ where $$
\Lambda_{\mathbf{k}} :=\{\mathbf{k}+\mathbf{m}=(k_i + m_i )_{i\in
\overline{\sigma}}\in \N^{n-d}: \; \sum_{i\in \overline{\sigma}}
a_{i} m_i \in \Z A_{\sigma} \}$$ is annihilated by the operators
(\ref{Toric-operators}) and (\ref{Euler-operators}). Notice that
$\varphi_{v^{\mathbf{k}}}$ is zero if and only if for all
$\mathbf{m}\in \Lambda_{\mathbf{k}} $, $A_{\sigma}^{-1}(\beta -
\sum_{i\notin \sigma } (k_i + m_i ) a_i)$ has at least one negative
integer coordinate.

\vspace{.3cm}

Let us consider the lattice $\Z \sigma = \Z A_{\sigma} =\sum_{i\in
\sigma} \Z a_i $ contained in $\Z A$.

\begin{lema}\label{equiv-lattice}
The following statements are equivalent for all $\mathbf{k},
\mathbf{k}' \in \Z^{n-d}$:

\begin{enumerate}
\item[1)] $v^{\mathbf{k}}-v^{\mathbf{k}'} \in \Z^n$

\item[2)] $[A_{\overline{\sigma}}\mathbf{k}]=[
A_{\overline{\sigma}}\mathbf{k}']$ in $\Z A /\Z \sigma $.

\item[3)] $\Lambda_{\mathbf{k}} =\Lambda_{\mathbf{k}'}$.

\end{enumerate}
\end{lema}

\begin{lema}\label{cardinal-lattice}
We have the equality: $$\{\Lambda_{\mathbf{k}} :\; \mathbf{k} \in
\Z^{n-d} \} =\{\Lambda_{\mathbf{k}} :\; \mathbf{k} \in \N^{n-d} \}
$$ and the cardinal of this set is $[\Z A : \Z \sigma ]=\operatorname{vol}_{\Z A}(\Delta_{\sigma })$.
\end{lema}

\begin{proof}
The equality is clear because $ A_{\overline{\sigma}}\mathbf{c} \in
\Z \sigma$ for $\mathbf{c}=|\det(A_{\sigma} )| \cdot (1,\ldots ,1)
\in (\N^{\ast})^{n-d}$ and then for any $\mathbf{k} \in \Z^{n-d}$
there exists $\alpha \in \N$ such that $\mathbf{k}+ \alpha
\mathbf{c} \in \N^{n-d}$ and
$\Lambda_{\mathbf{k}}=\Lambda_{\mathbf{k}+\alpha \mathbf{c} }$.

\vspace{.3cm}

$\forall \overline{\lambda} \in \Z A /\Z \sigma$ there exists
$\mathbf{k}\in \Z^{n-d}$ with
$\overline{A_{\overline{\sigma}}\mathbf{k}} = \overline{\lambda} \in
\Z A /\Z \sigma$. Then by the equivalence of 2) and 3) in Lemma
\ref{equiv-lattice} we have that $\{ \Lambda_{\mathbf{k}} :\;
\mathbf{k} \in \Z^{n-d} \}$ has the same cardinal as the finite
group $\Z A / \Z \sigma$.
\end{proof}

\begin{nota}\label{disjoint-supports}
Recall that the support of a series $\sum_{v} c_v x^v$ is the set $$
\{v\in \C^n :\; c_v \neq 0\}.$$ Then, for all $\mathbf{k},
\mathbf{k}' \in \N^{n-d}$ such that $v^{\mathbf{k}}-v^{\mathbf{k}'}
\in \Z^n$ we have that
$\varphi_{v^{\mathbf{k}}}=\varphi_{v^{\mathbf{k}'}}$ and in other
case we have that $\varphi_{v^{\mathbf{k}}}$,
$\varphi_{v^{\mathbf{k}'}}$ have disjoint supports.
\end{nota}

\begin{nota}\label{partition}
One may consider $\mathbf{k}(1) , \ldots , \mathbf{k}(r) \in
\N^{n-d}$ such that $$\Z A /\Z \sigma = \{ [A_{\overline{\sigma}}
\mathbf{k}(i) ]: \; i=1,\ldots , r \}$$ with $r=[\Z A : \Z \sigma
]$. Then the set in Lemma \ref{cardinal-lattice} is equal to $\{\
\Lambda_{ \mathbf{k}(i) }: \; i=1,\ldots ,r \} $ and it determines a
partition of $\N^{n-d}$, i.e.,

\begin{enumerate}
\item[1)] $\Lambda_{\mathbf{k}(i)} \cap
\Lambda_{\mathbf{k}(j)}=\emptyset$ if $i\neq j$.

\item[2)] $ \cup_{i=1}^r \Lambda_{\mathbf{k}(i)}
=\N^{n-d}$.
\end{enumerate}
\end{nota}

We have described $\operatorname{vol}_{\Z A }(\Delta_{\sigma })$
formal solutions of $\M_A (\beta )$ along $Y_{\sigma}$ associated
with a simplex $\sigma$ having pairwise disjoint supports. Thus,
they are linearly independent if none of them is zero.

\begin{defi}\label{definition-generic}
$\beta \in \C^d$ is said to be generic if it runs in a Zariski open set.
$\beta$ is said to be very generic if it runs in a countable
intersection of Zariski open sets.
\end{defi}

\begin{nota}
For generic $\beta $ we have that $\beta$ is non-rank-jumping. For
very generic $\beta $ we have that $v^{\mathbf{k}}$ does not have
any negative integer coordinate for all $\mathbf{k} \in \N^{n-d}$
and that $\beta $ is generic. In particular, if $\beta$ is very
generic we have that $ \varphi_{v^{\mathbf{k}}} \neq 0$, $\forall
\mathbf{k}$. More precisely, non generic parameter vectors $\beta$
(resp. non very generic) lie in the complement of a hyperplane
arrangement (resp. a countable union of hyperplane arrangements)
that depends on $A$.
\end{nota}

These $\Gamma $--series are handled in \cite{SST} in such a way that
they are not zero for any $\beta \in \C^d$:

$$\phi_v :=\sum_{u\in N_v} \frac{[v]_{u_{-}}}{[v+u]_{u_{+}}}
x^{v+u}$$ where $v\in \C^n$ verifies $A v=\beta$ and $N_v = \{u\in
L_A : \; \operatorname{nsupp}(v+u)=\operatorname{nsupp}(v) \}$. Here
$\operatorname{nsupp}(w):=\{i\in \{1,\ldots , n\}: \; w_i \in \Z_{<0
}\}$ for $w\in \C^n$, $[v]_{u}=\prod_{i} [v_i ]_{u_i} $ and $[v_i
]_{u_i}=\prod_{j=1}^{u_i } (v_i -j +1)$ is the Pochhammer symbol for
$v_i \in \C$, $u_i \in \N$. When $v\in (\C \setminus \Z_{<0})^{n}$
we have:

$$\phi_v = \Gamma (v+1) \varphi_v .$$ Since $A v=\beta$ the series
$\phi_v$ is annihilated by the operators (\ref{Euler-operators}). It
is annihilated by the toric ideal $I_A$ if and only if the negative
support of $v$ is minimal, i.e., $\nexists u \in L_A :=\ker (A)\cap
\Z^n $ with $\operatorname{nsupp}(v+u)\subsetneq
\operatorname{nsupp}(v)$ (see \cite{SST}, Section 3.4.).

\vspace{.3cm}

Observe that any $u \in L_A$ has the form $(-\sum_{j \notin \sigma }
r_j A_{\sigma}^{-1}a_j , \mathbf{r})$ with $ \mathbf{r}=(r_j )_{j
\notin \sigma } \in \Z^{n-d}$ such that $A_{\overline{\sigma}
}\mathbf{r}= \sum_{j\notin \sigma } r_j a_j \in \Z \sigma$. Then we
can choose $\mathbf{k} \in \N^{n-d}$ such that $v^{\mathbf{k}}$ has
minimal negative support because we do not change the class of
$\sum_{j\notin \sigma} k_j a_j$ modulo $\Z A_{\sigma}$ when
replacing $\mathbf{k}$ by $\mathbf{k}+\mathbf{r}\in \N^{n-d}$. Then
the new series $\phi_{\sigma}^{\mathbf{k}} :=
\phi_{v^{\mathbf{k}}}\neq 0$ has the form:

$$\phi_{\sigma}^{\mathbf{k}} = \sum_{\mathbf{k}+ \mathbf{m}\in
S_{\mathbf{k}}}  \frac{[v^{\mathbf{k}}
]_{u(\mathbf{m})_{-}}}{[v^{\mathbf{k}}
+u(\mathbf{m})]_{u(\mathbf{m})_{+}}} \; x^{v^{\mathbf{k}}
+u(\mathbf{m})}$$ where
$$S_{\mathbf{k}} :=\{ \mathbf{k}+\mathbf{m}\in \Lambda_{\mathbf{k}}:
\;
\operatorname{nsupp}(v^{\mathbf{k}+\mathbf{m}})=\operatorname{nsupp}(v^{\mathbf{k}})
\}\subseteq \Lambda_{\mathbf{k}} $$ and
$u(\mathbf{m})=(-\sum_{i\notin \sigma}m_i A_{\sigma}^{-1} a_i ,
\mathbf{m})$ for $\mathbf{m}=(m_i )_{i\notin \sigma }\in \Z^{n-d}$.
It is clear that $\mathbf{k}+\mathbf{m} \in S_{\mathbf{k}}$ if and
only if $v^{\mathbf{k}+\mathbf{m}}\in v^{\mathbf{k}} +
N_{v^{\mathbf{k}}}$.

\vspace{.3cm}

\begin{nota}\label{disjoint-supports-2}
Using that $ S_{\mathbf{k}} \subseteq \Lambda_{\mathbf{k}}, \;
\forall \mathbf{k} \in \N^{n-d}$, and Remark \ref{partition} we have
that two series in $\{\phi_{\sigma}^{\mathbf{k}}: \; \mathbf{k}\in
\N^{n-d} \}$ are either equal up to multiplication by a nonzero
scalar or they have disjoint supports. Thus, the set
$\{\phi_{\sigma}^{\mathbf{k}}: \; \mathbf{k}\in \N^{n-d} \}$ has
$\operatorname{vol}_{\Z A }(\Delta_{\sigma} )$ linearly independent
formal series solutions of $\M_A (\beta )$ along $Y_{\sigma}$ at any
point of $Y_{\sigma}\cap \{x_j \neq 0: \; j\in \sigma \}$ for all
$\beta \in \C^d$.
\end{nota}

\begin{ejem}\label{ejem1}
Let $A=(a_1 \; a_2 \; a_3 )\in \Z^{2\times 3}$ be the matrix with
columns: $$a_1 =\left(\begin{array}{c}
                       1 \\
                       0
                     \end{array}
\right)\; \; a_2 =\left(\begin{array}{c}
                       0 \\
                       2
                     \end{array}
\right)\; \; a_3 =\left(\begin{array}{c}
                       3 \\
                       1
                     \end{array}\right)$$ The kernel of $A$ is
                     generated by $u=(6,1,-2)$ and so $L_A =\Z u$. Then the hypergeometric system associated with $A$ and $\beta \in
\C^2$ is generated by the differential operators:
$$\Box_{u}=\partial_1^6 \partial_2 -\partial_3^2,\; E_1 -\beta_1 =x_1
\partial_1 + 3 x_3 \partial_3 -\beta_1, \;E_2 -\beta_2 = 2 x_2
\partial_2 +x_3 \partial_3 -\beta_2.$$ In this example $\Z A=\Z^2$,
$A$ is pointed and $\sigma =\{ 1,2\}$ is a simplex with normalized
volume $\operatorname{vol}_{\Z A}(\Delta_{\sigma})=|\det (A_{\sigma
})|=2$ (see Figure 1).

\setlength{\unitlength}{8mm}
$$\begin{picture}(-2.5,6)(10,2)
\put(5,3){\vector(0,1){5}}
\put(2,5){\vector(1,0){6}}
\put(6,5){\line(-1,2){1}}
\put(6,5){\makebox(0,0){$\bullet$}}
\put(5,7){\makebox(0,0){$\bullet$}}
\put(8,6){\makebox(0,0){$\bullet$}}
\put(6,5){\makebox(0.8,-0.4){$a_1$}}
\put(5,7){\makebox(0.9,0){$a_2$}}
\put(8,6){\makebox(0.8,-0.4){$a_3$}} \put(5,2){\makebox(0,0){Figure
1}}
\end{picture}$$ Two convenient vectors associated with $\sigma$ are $$v^{0}=(\beta_1
,\beta_2 /2 ,0) \mbox{ and } v^{1}=(\beta_1 -3 ,(\beta_2 -1) /2 ,
1).$$ The associated series:

$$\phi_{v^{0}}=\sum_{m\geq 0} \frac{[\beta_1 ]_{6m} [\beta_2 /2 ]_m }{(2m)!} x_1^{\beta_1 -6 m} x_2^{\beta_2 /2  -m} x_3^{2m} $$
and $$\phi_{v^{1}}=\sum_{m\geq 0} \frac{[\beta_1 -3]_{6m} [(\beta_2
-1) /2 ]_m }{(2m + 1)!} x_1^{\beta_1 - 3 -6 m} x_2^{(\beta_2 -1) /2
-m} x_3^{1+2m}$$ are formal series  along $Y_{\sigma}=\{x_3 =0\}$ at
any point of $Y_{\sigma }\cap \{x_1 x_2 \neq 0 \} $ that are
annihilated by the Euler operators $E_1 -\beta_1$, $E_2 -\beta_2$
because $A v^{k}=\beta$ and by the toric operator $\Box_{u}$ since
$v^{k}$ has minimal negative support for all $\beta \in \C^2$ for
$k=0,1$.
\end{ejem}

The following Lemma is very related with Lemma 1 in
\cite{Ohara-Takayama} (see also \cite[Proposition 1, Section
1.1.]{GKZ} and \cite[ Proposition 5]{PST}).

\begin{lema}\label{equiv-Gevrey}
Assume that $\{b_i \}_{i={d+1}}^{n}$ is a set of vectors in
$\Q^{d}\times \N^{n-d}$, $\mathbf{k}\in \Z^{n-d}$. Let us denote
$u(\mathbf{m})=\sum_{i=d+1}^n m_i b_i$ and consider a set
$D_{\mathbf{k}} \subseteq \{\mathbf{k}+\mathbf{m} \in \N^{n-d}: \;
u(\mathbf{m})\in \Z^n \}$ and a vector $v\in \C^n$ such that
$\operatorname{nsupp}(v+ u(\mathbf{m}))=\operatorname{nsupp}(v)$ for
any $\mathbf{m} \in D_{\mathbf{k}} - \mathbf{k}$. Then for all
$\mathbf{s}\in \R^{n-d}$ the following statements are equivalent:

\begin{enumerate}
\item[1)] $\displaystyle \sum_{\mathbf{k}+\mathbf{m} \in D_{\mathbf{k}}} \dfrac{[v]_{u(\mathbf{m})_{-}}}{[v+u(\mathbf{m})]_{u(\mathbf{m})_{+}}}
y^{\mathbf{k}+\mathbf{m}}$ is Gevrey of multi-order $\mathbf{s}$
along $y=0$.

\item[2)] $\displaystyle \sum_{\mathbf{k}+\mathbf{m} \in D_{\mathbf{k}} } \dfrac{u(\mathbf{m})_{-}!}{u(\mathbf{m})_{+}!}
y^{\mathbf{k}+\mathbf{m}}$ is Gevrey of multi-order $\mathbf{s}$
along $y=0$.

\item[3)] $\displaystyle \sum_{\mathbf{k}+\mathbf{m} \in D_{\mathbf{k}} } \prod_{j=d+1}^n (k_j + m_j )!^{-|b_j|}
y^{\mathbf{k}+\mathbf{m}}$ is Gevrey of multi-order $\mathbf{s}$
along $y=0$.
\end{enumerate}

In particular, for $\mathbf{s}=(s_{d+1},\ldots ,s_n )$ with $s_i =
1- |b_i|$, $i=d+1.\ldots,n$, 1),2) and 3) are satisfied. Moreover,
1), 2) and 3) are also equivalent if we write {\em order } $s$
instead of {\em multi-order} $\mathbf{s}$ and all these series are
Gevrey of order $s=\displaystyle \max_{i} \{ 1-|b_i | \}$.
\end{lema}

\begin{proof}
$\forall \alpha \in \C$, $\forall m \in \N$ with $[\alpha]_m \neq 0$
there exists $C, D >0$ such that:

\begin{equation}
C^m | [\alpha ]_m | \leq m! \leq |[\alpha ]_m | D^m \label{formula1}
\end{equation}

For the proof of (\ref{formula1}) it is enough to consider $c_m :=
[\alpha ]_m / m!$ and see that $\lim_{m \rightarrow \infty }
|c_{m+1}/c_m |=1$. The proof for (\ref{formula1}) with $\alpha + m$
instead of $\alpha$ is analogous. It follows that 1) and 2) are
equivalent.

\vspace{.3cm}

We can use Stirling's formula $m! \sim \sqrt{2 \pi m} (m/e)^m$ in
order to prove that $\forall m\in \N$, $\forall q \in \Q_{+}$ with
$q m \in \N$, there exist $C',D'>0$ verifying:

\begin{equation}
(C')^m m!^q \leq (qm)! \leq (D')^m m!^q \label{formula4}
\end{equation}

Take $\lambda \in \N^{\ast}$ such that $\lambda b_i \in \Z^n$ for
all $i=d+1,\ldots ,n$. Then by (\ref{formula4}) we have 2) if and
only if the series $\displaystyle \sum_{\mathbf{k}+\mathbf{m} \in
D_{\mathbf{k}} } \dfrac{((\lambda u(\mathbf{m})_{-})!)^{1/ \lambda
}}{((\lambda u(\mathbf{m})_{+})!)^{1/ \lambda }}
y^{\mathbf{k}+\mathbf{m}}$ is Gevrey of multi-order $\mathbf{s}$
along $y=0$.

\vspace{.3cm}

For the rest of the proof we assume for simplicity that
$D_{\mathbf{k}} \subseteq (\mathbf{k} + \N^{n-d})\cap \N^{n-d}$. The
equivalence of 2) and 3) can be proven without this assumption but
it is necessary to distinguish more cases.

\vspace{.3cm}

Observe that $\lambda u(m)_{+}- \lambda u(m)_{-}= \sum_{i=d+1}^n
\lambda (b_{i})_{+}  m_i  - \sum_{i=d+1}^n  \lambda (b_{i})_{-} m_i$
and $u(m)_{+},u(m)_{-}, \sum_{i=d+1}^n \lambda (b_{i})_{+} m_i ,
\sum_{i=d+1}^n \lambda (b_{i})_{-} m_i \in \N^n$. However,
$u(m)_{+}, u(m)_{-}$ have disjoint supports while, in general, $
\sum_{i=d+1}^n \lambda (b_{i})_{+} m_i $ and $ \sum_{i=d+1}^n
\lambda (b_{i})_{-} m_i $ do not.

\vspace{.3cm}

On the other hand, for all $ m , n \in \N$ with $n\leq m$ we have
that:

\begin{equation}
(m-n)! \leq  \frac{m!}{n!} \leq 2^m (m-n)! \label{formula2}
\end{equation}

Then by (\ref{formula2}) we have 2) if and only if

\begin{equation} \displaystyle \sum_{\mathbf{k}+\mathbf{m} \in S}
\left(\dfrac{(\sum_{i=d+1}^n \lambda (b_{i})_{-}  m_i
)!}{(\sum_{i=d+1}^n \lambda (b_{i} )_{+} m_i )!}\right)^{1/\lambda}
y^{\mathbf{k}+\mathbf{m}}\label{series-intermediate2}
\end{equation} is Gevrey of multi-order $\mathbf{s}$ along $y=0$.

\vspace{.3cm}

If we replace $m$ by $m+n$ in (\ref{formula2}) and multiply by $n!$
we obtain a formula that can be generalized by induction. We obtain
that $\forall m_{d+1} ,\ldots , m_{n} \in \N$ there exist $C'' , D''
>0$ such that:

\begin{equation}
(C'')^{\sum_{j} m_j} \prod_{i} m_i ! \leq  (\sum m_i)! \leq
(D'')^{\sum_{j} m_j} \prod_{i} m_i !  \label{formula3}
\end{equation}

A combination of (\ref{formula3}) and (\ref{formula4}) proves that
3) holds if and only if (\ref{series-intermediate2}) is Gevrey of
multi-order $\mathbf{s}$ along $y=0$.

\vspace{.3cm}

Finally, it is clear that 3) is true for $s_i = 1- |b_i|$, $i=d+1,
\ldots ,n$.
\end{proof}

\begin{nota}
From the proof of the equivalence of 1) and 2) in Lemma
\ref{equiv-Gevrey} it can be deduced that after applying
$\rho_{\mathbf{s}}$ to the series in 1) there exists an open set $W$
such that this series converges in $W$ and $0\in W$ does not depend
on $v$ but in $D_k -k$.
\end{nota}

Consider $\mathbf{s}=(s_j)_{j\notin \sigma}$ with $$s_j :=
|A_{\sigma}^{-1}a_j|,\; j\notin \sigma$$ and $s=\max_{i} \{ s_i \}$
throughout this section.

\begin{lema}\label{lema-Gevrey}
For all $\mathbf{k}\in \N^{n-d}$ the series

$$\psi_{\sigma}^{\mathbf{k}} := \sum_{\mathbf{k}+\mathbf{m}\in S_{k}}
\frac{[v^{k}]_{u(m)_{-}}}{[v^{\mathbf{k}}
+u(\mathbf{m})]_{u(\mathbf{m})_{+}}} y^{\mathbf{k}+\mathbf{m}}$$ is
Gevrey of multi-order $\mathbf{s}$ along $y= \mathbf{0}\in
\C^{n-d}$. Moreover, if $\beta$ is very generic then it has Gevrey
index $s$ along $y=\mathbf{0}\in \C^{n-d}$.
\end{lema}

\begin{proof}
It follows from Lemma \ref{equiv-Gevrey} (if we take $b_{d+i}$ equal
to the $i$-th column of $B_{\sigma}$,
$D_{\mathbf{k}}=S_{\mathbf{k}}$ and $v=v^{\mathbf{k}}$) that
$\psi_{\sigma}^{\mathbf{k}}$ is Gevrey of multi-order $\mathbf{s}$
along $y = \mathbf{0}$.

\vspace{.3cm}

If $\beta$ is very generic we have that $S_{\mathbf{k}} =
\Lambda_{\mathbf{k}}$ and it is obvious that the series in 3) of
Lemma \ref{equiv-Gevrey} has Gevrey index $s$ in this case.
\end{proof}

\begin{cor}
The series $\phi_{\sigma}^{\mathbf{k}} $ is Gevrey of multi-order
$\mathbf{s}=(s_j)_{j\notin \sigma}$ along $Y_{\sigma}$ at any point
of $Y_{\sigma} \cap \{ x_{i}\neq 0 : \; i\in \sigma \}$. If $\beta$
is very generic then it is Gevrey with index $s$ along $Y_{\sigma}$.
\end{cor}

\begin{proof}
If we take $y=(y_j)_{j\notin \sigma}$ with $y_j :=
x_{\sigma}^{-A_{\sigma}^{-1}a_{j}} x_{j}$, $j\notin \sigma$, then
$\phi_{\sigma}^{\mathbf{k}} (x)=x_{\sigma}^{A_{\sigma}^{-1}\beta}
\psi_{\sigma}^{\mathbf{k}} (y)$ and the result follows from Lemma
\ref{lema-Gevrey}.
\end{proof}

\begin{ejem}
\textbf{(Continuation of Example \ref{ejem1})} We have that $$\rho_s
(\phi_{v^{0}})=x_1^{\beta_1} x_2^{\beta_2 /2 }\sum_{m\geq 0}
\frac{[\beta_1 ]_{6m} [\beta_2 /2 ]_m }{(2m)!^{s}}
\left(\frac{x_3^2}{x_1^6 x_2}\right)^{m}$$ It is easy to see that
$\rho_s (\phi_{v^{0}})$ has a nonempty domain of convergence if and
only if $s\geq 7/2$ when $\beta_1 , \beta_2 /2 \notin \N$ (use
D'Alembert criterion for the series in one variable $y= x_3^2
/(x_1^6 x_2 )$). Then $\phi_{v^{0}}$ is a Gevrey series solution of
$\M_A (\beta )$ with index $s=7/2$ along $Y_{\sigma}=\{x_3 =0\}$ at
any point of $Y_{\sigma}\cap \{x_1 x_2 \neq 0\}$. Nevertheless,
$\phi_{v^{0}}$ is a finite sum if either $\beta_1 \in \N$ or
$\beta_2 /2 \in \N$ and so it has the same convergence domain as the
(multi-valued) function $x_1^{\beta_1} x_2^{\beta_2 /2 }$. If both
$\beta_1 , \beta_2 /2 \in \N$, then $\phi_{v^{0}}$ is a polynomial.

\vspace{.3cm}

Analogously, $\phi_{v^{1}}$ is a Gevrey series solution of order
$s=7/2$ along $Y_{\sigma}$ at any point of $Y_{\sigma }\cap \{x_1
x_2 \neq 0 \}$. It has Gevrey index $s=7/2$ if $\beta_1 -3 ,
(\beta_2 -1)/2 \notin \N$ and it is convergent in other case.

\vspace{.3cm}

Notice that $s=7/2$ is the unique algebraic slope of $\M_A (\beta )$
along $Y_{\sigma}=\{x_3 =0\}$ at $\mathbf{0}\in \C^3$ (see \cite{SW}
or \cite{hartillo_rmi}).
\end{ejem}

\vspace{.3cm}

The convergence domain of
$\rho_{\mathbf{s}}^{\emptyset}(\psi_{\sigma}^{\mathbf{k}})$ contains
$\{y\in \C^{n-d}: \; |y_j| < R , \; j \notin \sigma\}$ for certain
$R>0$. In particular,
$\rho_{\mathbf{s}}^{\sigma}(\phi_{v^{\mathbf{k}}})$ converges in
$$\{x\in \C^n : \; \prod_{i\in \sigma} x_{i} \neq 0 ,
\; |x_{j}|< R |x_{\sigma}^{A_{\sigma}^{-1}a_{j}}|, \; \forall j
\notin \sigma \}.$$

\noindent The unique hyperplane that contains $\sigma$ is
$$H_{\sigma}= \{\mathbf{y} \in \R^d : \; |A_{\sigma}^{-1}\mathbf{y}|=1 \}
$$ and we denote by $H_{\sigma}^{-}:= \{
\mathbf{y} \in \R^d: \;|A_{\sigma}^{-1}\mathbf{y}|<1\}$ (resp. by
$H_{\sigma}^{+}:= \{ \mathbf{y} \in \R^d:
\;|A_{\sigma}^{-1}\mathbf{y}|>1\}$) the open affine half-space that
contains (resp. does not contain) the origin $\mathbf{0}\in \R^d$.

\vspace{.3cm}

\noindent Recall that $\mathbf{s}=(s_i )_{i\notin \sigma }$ where
$s_{i} = |A_{\sigma}^{-1}a_i|$ is the unique rational number such
that $a_i /s_i \in H_{\sigma}$. Moreover, $s_{i}>1$ (resp. $s_i <1$)
if and only if $a_i \in H_{\sigma}^{+}$ (resp. $a_i \in
H_{\sigma}^{-}$). Taking the set $$\tau=\{i : \; a_i \notin
H_{\sigma}^{+}\}$$ and $\mathbf{s}'=(s_{i})_{i\notin \tau}$ we have
that $\rho_{\mathbf{s}'}^{\tau }(\phi_{v^{\mathbf{k}}})$ converges
in the open set $$ U_{\sigma} ':= \{x\in \C^n : \; \prod_{i\in
\sigma} x_{i} \neq 0 , \; |x_{j}|< R
|x_{\sigma}^{A_{\sigma}^{-1}a_{j}}|, \; \forall a_j \in
(H_{\sigma}\setminus \sigma )\cup H_{\sigma}^{+} \}.$$ This implies
that $\phi_{v^{\mathbf{k}}}$ is Gevrey of multi-order $\mathbf{s}'$
along $Y_{\tau}$ at any point of $U_{\sigma}'\cap Y_{\tau}$. Then,
if we consider \begin{equation} U_{\sigma}:= \{x\in \C^n : \;
\prod_{i\in \sigma} x_{i} \neq 0 , \; |x_{j}|< R
|x_{\sigma}^{A_{\sigma}^{-1}a_{j}}|, \; \forall a_j \in
H_{\sigma}\setminus \sigma \} \label{domain}
\end{equation} the following result is obtained.

\begin{teo}\label{GevreyI}
For any set $\varsigma $ with $\sigma \subseteq \varsigma \subseteq
\tau$ the series $$\phi_{\sigma}^{\mathbf{k}}=\sum_{\mathbf{k}+
\mathbf{m}\in S_{\mathbf{k}}}  \frac{[v^{\mathbf{k}}
]_{u(\mathbf{m})_{-}}}{[v^{\mathbf{k}}
+u(\mathbf{m})]_{u(\mathbf{m})_{+}}} \; x^{v^{\mathbf{k}}
+u(\mathbf{m})}$$ is a Gevrey series solution of $\M_A (\beta )$ of
order $s=\operatorname{max} \{s_i =|A_{\sigma}^{-1}a_i| : \; i\notin
\sigma \}$ along $Y_{\varsigma}$ at any point of $Y_{\varsigma} \cap
U_{\sigma}$. If $\beta$ is very generic then $s$ is its Gevrey
index.
\end{teo}

\begin{nota}
If $H_{\sigma}\cap \{a_i : \; i=1,\ldots , n \}=\sigma$ then
$U_{\sigma }= \{\prod_{i\in \sigma } x_i \neq 0\}$.
\end{nota}

\begin{nota}
Recall that in Theorem \ref{GevreyI} the vector
$v^{\mathbf{k}}=(A_{\sigma}^{-1}(\beta
-A_{\overline{\sigma}}\mathbf{k}),\mathbf{k})$ has minimal negative
support because we have chosen $\mathbf{k}\in \Lambda_{\mathbf{k}}$
this way. This guarantees that $\phi_{v^{\mathbf{k}}}$ is a solution
of $\M_A (\beta )$. However, this series is also Gevrey of order $s$
when $\mathbf{k}$ does not satisfy this condition.
\end{nota}

\section{Slopes of $\M_A (\beta )$ associated with a
simplex}\label{slopes-simplex-section}

In the context of Section \ref{construction-Gevrey} we fix a simplex
$\sigma \subseteq A$ with $\det (A_{\sigma })\neq 0$ and consider
$\mathbf{s}=(s_i )_{i\notin \sigma }$ where $s_{i} =
|A_{\sigma}^{-1}a_i|$. We consider $\tau =\{j :\; a_j \notin
H_{\sigma}^{+} \}\supseteq \sigma $ and the coordinate subspace
$Y_{\tau}=\{x_{j}=0: \; j\notin \tau \}$ in this section.

\vspace{.3cm}

Our purpose here is to construct one nonzero Gevrey series solution
of $\M_A (\beta ) $ in
$(\cO_{X|Y_{\tau}}(s)/\cO_{X|Y_{\tau}}(<s))_p$ for $p\in
Y_{\tau}\cap U_{\sigma}$ with support contained in the set
$\Lambda_{\mathbf{k}}\subseteq \N^{n-d} $ in the partition of
$\N^{n-d}$ (see Remark \ref{partition}) for all $\beta \in \C^d$. In
particular we will prove the following result:

\begin{prop}\label{dim-gevrey-simplex}
For $s=\max \{s_{i} = |A_{\sigma}^{-1}a_i|: \; i\notin \sigma \}$,
for all $p\in Y_{\tau} \cap U_{\sigma}$ and for all $\beta\in \C^d$:

$$\dim (\mathcal{H}om_{\D}(\M_A (\beta
),\cO_{X|Y_{\tau}}(s)/\cO_{X|Y_{\tau}}(<s)))_p \geq
\operatorname{vol}_{\Z A }(\Delta_{\sigma} ).$$
\end{prop}

As a consequence of Proposition \ref{dim-gevrey-simplex} and Lemma
\ref{lema-slope}, we obtain the following result that justifies the
name of this section:

\begin{cor}\label{analytic-slope} If $Y_{\tau}$ is a coordinate hyperplane (equivalently, the cardinal of $\tau$ is $n-1$) and $s=|A_{\sigma}^{-1}a_{\overline{\tau}}|>1$ then $s$ is an analytic slope
of $\M_A (\beta )$ along $Y_{\tau}$ at any point in the closure of
$Y_{\tau}\cap U_{\sigma }$.
\end{cor}

\begin{nota}
By Theorem \ref{GevreyI} we only need Lemma \ref{lema-slope} for the
proof of Corollary \ref{analytic-slope} if $\beta$ is not very
generic.
\end{nota}

\begin{nota}
Observe that $\mathbf{0}$ is in the closure of $Y_{\tau}\cap
U_{\sigma}$. However, by Remark \ref{nota-slope} we have that $s$ is
a slope along $Y_{\tau}$ at any point of $Y_{\tau}$.
\end{nota}

\begin{nota}
Corollary \ref{analytic-slope} uses that $Y_{\tau}$ has codimension
one because the analytic slopes along $Y_{\tau}$ are not defined if
the codimension is greater. However, a good definition for the
analytic slopes of a holonomic $\D$-module along a variety $Z$ of
codimension greater than $1$ should include the indices of the
Gevrey series solutions of $\M_A (\beta )$ along $Z$. One reason is
that $\{\mathcal{H}^0 (\operatorname{Irr}_Z^{(s)}(\M))\}_{s\geq 1}$
determines a Gevrey filtration of $\mathcal{H}^0
(\operatorname{Irr}_Z(\M))$, although
$\{\operatorname{Irr}_Z^{(s)}(\M)\}_{s\geq 1}$ does not determine in
general a filtration of $\operatorname{Irr}_Z(\M)$.

\end{nota}

Let us proceed with the construction of the announced series and the
proof of Proposition \ref{dim-gevrey-simplex}.

\vspace{.3cm}

We identify $\mathbf{k}+\mathbf{m} \in \N^{n-d}$ with
$v^{\mathbf{k}+\mathbf{m}}=(A_{\sigma}^{-1}(\beta -
A_{\overline{\sigma}}(\mathbf{k}+\mathbf{m})),
\mathbf{k}+\mathbf{m}) \in \C^{d}\times \N^{n-d}$ and establish a
partition of $\Lambda_{\mathbf{k}} $ in terms of the negative
support of the vector $v^{\mathbf{k}+\mathbf{m}} \in \C^{d}\times
\N^{n-d}$ as follows. For any subset $\eta \subseteq \sigma$ set:

$$\Lambda_{\mathbf{k},\eta }:=\{\mathbf{k}+\mathbf{m} \in \Lambda_{\mathbf{k}} :\; \mbox{nsupp}(A_{\sigma}^{-1}(\beta
-A_{\overline{\sigma}}(\mathbf{k}+\mathbf{m})))=\eta \}.$$ Consider
the set $$\Omega_{\mathbf{k}} := \{ \eta \subseteq \sigma : \;
\Lambda_{\mathbf{k},\eta} \neq \emptyset \}.$$

Then it is clear that $\{\Lambda_{\mathbf{k},\eta }: \eta \in
\Omega_{\mathbf{k}}\} $ is a partition of $\Lambda_{\mathbf{k}}$.
Moreover $\Lambda_{\mathbf{k},\eta }$ is the intersection of a
polytope with $\Lambda_{\mathbf{k}}$ because the conditions
$$\mbox{nsupp}(A_{\sigma}^{-1}(\beta
-A_{\overline{\sigma}}(\mathbf{k}+\mathbf{m})))=\eta$$ are
equivalent to inequalities of type:

$$(A_{\sigma}^{-1}(\beta
-A_{\overline{\sigma}}(\mathbf{k}+\mathbf{m})))_i < 0  $$ for $i\in
\eta$ and
$$(A_{\sigma}^{-1}(\beta
-A_{\overline{\sigma}}(\mathbf{k}+\mathbf{m})))_j \geq 0 $$ for $j
\notin \eta$ such that $(A_{\sigma}^{-1}(\beta
-A_{\overline{\sigma}}\mathbf{k}))_j \in \Z$.

\vspace{.3cm}

For any $\eta \in \Omega_{\mathbf{k}}$ the series
$\phi_{v^{\mathbf{k}+\mathbf{m}}}$ for $ \mathbf{k}+\mathbf{m} \in
\Lambda_{\mathbf{k},\eta} $ depends on $\Lambda_{\mathbf{k},\eta} $
but not on $\mathbf{k}+\mathbf{m}\in \Lambda_{\mathbf{k},\eta} $ up
to multiplication by nonzero scalars. Let us fix any $
\widetilde{\mathbf{k}} \in \Lambda_{\mathbf{k},\eta}$ and define:

$$\phi_{\mathbf{k},\eta}:=\phi_{v^{\widetilde{\mathbf{k}}}}.$$ Observe
that the support of the series $\phi_{\mathbf{k},\eta}$ is:
$$\mbox{supp}(\phi_{\mathbf{k},\eta})=\{v^{\mathbf{k}+\mathbf{m}} :
\; \mathbf{k}+\mathbf{m} \in \Lambda_{\mathbf{k},\eta }\}.$$

\vspace{.3cm}

If the set $\Omega_{\mathbf{k}}$ has only one element $\eta$ then
$\Lambda_{\mathbf{k}} =\Lambda_{\mathbf{k},\eta }$ and the series
$\phi_{\sigma}^{\mathbf{k}}=\phi_{\mathbf{k},\eta}$ is a nonzero
Gevrey series solution of $\M_A (\beta) $ in
$\cO_{X|Y_{\tau}}(s)\setminus \cO_{X|Y_{\tau}}(<s)$ at any point of
$Y_{\tau}\cap U_{\sigma}$ (see Theorem \ref{GevreyI}).

\vspace{.3cm}

All the series in the finite set $\{ \phi_{\mathbf{k},\eta} :
\mathbf{k}\in \N^{n-d}, \eta \in \Omega_{\mathbf{k}} \} $ are Gevrey
series along $Y_{\sigma}=\{x_i = 0 : \; i\notin \sigma \}$ with
multi-order $\mathbf{s}$ at points of $Y_{\sigma}\cap \{x_j \neq 0:
\; j \in \sigma \}$. In fact, these series are Gevrey of order $s$
along $Y_{\tau}$ at any point of $Y_{\tau}\cap U_{\sigma }$ and they
are all annihilated by the Euler operators.

\vspace{.3cm}

For all $\eta\in \Omega_{\mathbf{k}}$, the support of the series
$\phi_{\mathbf{k}, \eta }$ is $\operatorname{supp}(\phi_{\mathbf{k},
\eta })=\{v^{\mathbf{k}+\mathbf{m}}: \; \mathbf{k}+\mathbf{m} \in
\Lambda_{\mathbf{k},\eta } \} $ and $\cup_{\eta \in
\Omega_{\mathbf{k}}} \Lambda_{\mathbf{k},\eta}
=\Lambda_{\mathbf{k}}$. Then there exists $\eta \in
\Omega_{\mathbf{k}}$ such that $\phi_{\mathbf{k}, \eta }\in
\cO_{X|Y_{\tau}}(s)$ has Gevrey index $s$. But a series $\phi_v $ is
annihilated by $I_A$ if and only if $v$ has minimal negative support
(see \cite{SST}, Section 3.4.) so if we take $\eta '\in
\Omega_{\mathbf{k}}$ with minimal cardinal then
$\phi_{\mathbf{k},\eta '}\in \cO_{X|Y_{\tau}}(s)$ is a solution of
$\M_A (\beta )$. In general, we cannot take $\eta = \eta '$.

\vspace{.3cm}

The following Lemma is the key of the proof of Proposition
\ref{dim-gevrey-simplex}.

\begin{lema}\label{series-indice-adecuado}
Consider an element $\eta$ of the set
$$\{ \eta' \in \Omega_{\mathbf{k}}: \;
\phi_{\mathbf{k},\eta '} \mbox{ has Gevrey index } s \}$$ with
minimal cardinal. Then $\Box_{u} (\phi_{\mathbf{k} ,\eta} ) \in
\cO_{X|Y_{\tau}}(<s)$ for all $u\in L_A$.
\end{lema}

\begin{proof}
Consider $\Lambda_{\mathbf{k},\eta}$ with $\eta$ as above and $u\in
L_A$. Then there exists $\mathbf{\widetilde{m}}\in \Z^{n-d}$ such
that
$u=(-A_{\sigma}^{-1}A_{\overline{\sigma}}\mathbf{\widetilde{m}},\mathbf{\widetilde{m}})$
and then $$\Box_u =
\partial_{\sigma}^{(A_{\sigma}^{-1}A_{\overline{\sigma}}\mathbf{\widetilde{m}})_{-}}
\partial_{\sigma}^{\mathbf{\widetilde{m}}_{+}}-\partial_{\sigma }^{(A_{\sigma}^{-1}A_{\overline{\sigma}}\mathbf{\widetilde{m}})_{+}}
\partial_{\sigma}^{\mathbf{\widetilde{m}}_{-}}.$$

On the other hand, the series $\phi_{\mathbf{k}, \eta }$ has the
form:

$$\phi_{\mathbf{k},\eta}=\sum_{\mathbf{k}+\mathbf{m}\in \Lambda_{\mathbf{k},\eta}} c_{\mathbf{k}+\mathbf{m}}
x_{\sigma}^{A_{\sigma}^{-1} ( \beta - A_{\overline{\sigma}}
(\mathbf{k}+\mathbf{m}))}
x_{\overline{\sigma}}^{\mathbf{k}+\mathbf{m}}$$ where
$c_{\mathbf{k}+\mathbf{m}}\in \C$ verifies that
$c_{\mathbf{k}+\mathbf{m}+\mathbf{\widetilde{m}}}/c_{\mathbf{k}+\mathbf{m}}$
is a rational function on $\mathbf{m}$ (recall that there exists
$\mathbf{\widetilde{k}}\in \Lambda_{k,\eta}$ such that
$c_{\mathbf{k}+\mathbf{m}}=\frac{[v^{\mathbf{\widetilde{k}}}]_{u(\mathbf{k}-\mathbf{\widetilde{k}+\mathbf{m}})_{-}}}{[v^{\mathbf{k}+\mathbf{m}}]_{u(\mathbf{k}-\mathbf{\widetilde{k}+\mathbf{m}})_{+}}}$
by definition of $\phi_{\mathbf{k},\eta}$).

\vspace{.3cm}

A monomial $x^{v^{\mathbf{k}+\mathbf{m}}-u_{-}}=
x^{v^{\mathbf{k}+\mathbf{m}+\mathbf{\widetilde{m}}}-u_{+}}$
appearing in $\Box_u (\phi_{\mathbf{k},\eta})$ comes from the
monomials $x^{v^{\mathbf{k}+\mathbf{m}}}$ and
$x^{v^{\mathbf{k}+\mathbf{m}+\mathbf{\mathbf{\widetilde{m}}}}}$
after one applies $\partial^{u_{-}}$ and $\partial^{u_{+}}$
respectively.

\vspace{.3cm}

If $\mathbf{k}+\mathbf{m} ,
\mathbf{k}+\mathbf{m}+\mathbf{\widetilde{m}} \in
\Lambda_{\mathbf{k},\eta}$ then the monomial
$x^{v^{\mathbf{k}+\mathbf{m}}-u_{-}}$ appears in
$\partial^{u_{-}}(\phi_{\mathbf{k},\eta})$ and
$\partial^{u_{+}}(\phi_{\mathbf{k},\eta})$ with the same
coefficients so it doesn't appear in the difference.

\vspace{.3cm}

If $\mathbf{k}+\mathbf{m} \in \Lambda_{\mathbf{k},\eta}$ but
$\mathbf{k}+\mathbf{m}+\mathbf{\widetilde{m}} \notin
\Lambda_{\mathbf{k},\eta}$ (the case $\mathbf{k}+\mathbf{m} \notin
\Lambda_{\mathbf{k},\eta}$ but
$\mathbf{k}+\mathbf{m}+\mathbf{\widetilde{m}} \in
\Lambda_{\mathbf{k},\eta}$ is analogous), we can distinguish two
cases:

\begin{enumerate}
\item[1)] There exists $i$ such that $v^{\mathbf{k}+\mathbf{m}}_{i}\in
\N$ but $v^{\mathbf{k}+\mathbf{m}+\mathbf{\widetilde{m}}}_{i}<0$ so
$u_i =
v^{\mathbf{k}+\mathbf{m}+\mathbf{\widetilde{m}}}_{i}-v^{\mathbf{k}+\mathbf{m}
}_{i}<0$. Then $\partial^{u_{-}}(x^{v^{\mathbf{k}+\mathbf{m}}})=0$
and $x^{v^{\mathbf{k}+\mathbf{m}}-u_{-}}$ does not appear in $\Box_u
(\phi_{\mathbf{k},\eta })$.

\item[2)] We have $\mbox{nsupp}(v^{\mathbf{k}+\mathbf{m}+\mathbf{\widetilde{m}}})=\varsigma \subsetneq
\mbox{nsupp}(v^{\mathbf{k}+\mathbf{m}})=\eta$. Then
$[v^{\mathbf{k}+\mathbf{m}}]_{u_{-}}\neq 0$ and the coefficient of
$x^{v^{\mathbf{k}+\mathbf{m}}-u_{-}}$ in $\Box_u
(\phi_{\mathbf{k},\eta})$ is $c_{\mathbf{k}+\mathbf{m}}
[v^{\mathbf{k}+\mathbf{m}}]_{u_{-}}\neq 0$. Furthermore,
$\mathbf{k}+\mathbf{m}+\mathbf{\widetilde{m}} \in
\Lambda_{\mathbf{k},\varsigma}$ with $\varsigma \in
\Omega_{\mathbf{k}}$ such that $\phi_{\mathbf{k},\varsigma }$ is
Gevrey of index $s' < s$ because we chose $\eta$ that way.
\end{enumerate}

By 1), 2) and the analogous cases when $\mathbf{k}+\mathbf{m} \notin
\Lambda_{\mathbf{k},\eta}$ but
$\mathbf{k}+\mathbf{m}+\mathbf{\widetilde{m}} \in
\Lambda_{\mathbf{k},\eta}$, we have:

$$\Box_{u}(\phi_{\mathbf{k},\eta})=\displaystyle \sum_{\varsigma '}
\sum_{\stackrel{\mathbf{k}+\mathbf{m}+\mathbf{\widetilde{m}} \in
\Lambda_{\mathbf{k},\eta}}{\mathbf{k}+\mathbf{m}\in
\Lambda_{\mathbf{k},\varsigma '}}}
c_{\mathbf{k}+\mathbf{m}+\mathbf{\widetilde{m}}}
[v^{\mathbf{k}+\mathbf{m}+\mathbf{\widetilde{m}}}]_{u_{+}}
x^{v^{\mathbf{k}+\mathbf{m}+\mathbf{\widetilde{m}}}-u_{+}} -$$

\begin{equation}
- \sum_{\varsigma} \sum_{\stackrel{\mathbf{k}+\mathbf{m} \in
\Lambda_{\mathbf{k},\eta}}{\mathbf{k}+\mathbf{m}+\mathbf{\widetilde{m}}\in
\Lambda_{\mathbf{k},\varsigma}}} c_{\mathbf{k}+\mathbf{m}}
[v^{\mathbf{k}+\mathbf{m}}]_{u_{-}}
x^{v^{\mathbf{k}+\mathbf{m}}-u_{-}}
\end{equation} Here, $\varsigma , \varsigma ' \subseteq \eta $ varies in a subset of the finite set
$\Omega_{\mathbf{k}}$ whose elements $\varsigma ''$ verify that the
series $\phi_{\mathbf{k},\varsigma ''  }$ has Gevrey index $s''< s$.
Let us denote by $\widetilde{s}<s$ the maximum of these $s''$.

\vspace{.3cm}

Since
$c_{\mathbf{k}+\mathbf{m}+\mathbf{\widetilde{m}}}/c_{\mathbf{k}+\mathbf{m}}
$, $[v^{\mathbf{k}+\mathbf{m}}]_{u_{-}}$ and
$[v^{\mathbf{k}+\mathbf{m}+\mathbf{\widetilde{m}}}]_{u_{+}}$ are
rational functions on $\mathbf{m}$ the series
$\Box_{u}(\phi_{\mathbf{k},\eta})$ has Gevrey index at most the
maximum of the Gevrey index of the series

$$ \sum_{\varsigma '}
\sum_{\stackrel{\mathbf{k}+\mathbf{m}+\mathbf{\widetilde{m}} \in
\Lambda_{\mathbf{k},\eta}}{\mathbf{k}+\mathbf{m}\in
\Lambda_{\mathbf{k},\varsigma '}}} c_{\mathbf{k}+\mathbf{m}}
x^{v^{\mathbf{k}+\mathbf{m}+\mathbf{\widetilde{m}}}-u_{+}}
$$ $$ \sum_{\varsigma} \sum_{\stackrel{\mathbf{k}+\mathbf{m} \in
\Lambda_{\mathbf{k},\eta}}{\mathbf{k}+\mathbf{m}+\mathbf{\widetilde{m}}\in
\Lambda_{\mathbf{k},\varsigma}}}
c_{\mathbf{k}+\mathbf{m}+\mathbf{\widetilde{m}}}
x^{v^{\mathbf{k}+\mathbf{m}}-u_{-}}
$$ which is at most $\widetilde{s} < s$.

\vspace{.3cm}

It follows that $I_A (\phi_{\mathbf{k},\eta })\in
\cO_{X|Y_{\tau}}(<s)$ while $\phi_{\mathbf{k},\eta}$ has Gevrey
index $s$.
\end{proof}

Moreover the classes of the series $\{\phi_{\mathbf{k}
,\eta_{\mathbf{k}}}: \; \mathbf{k}  \in \N^{n-d} \}$ (with
$\eta_{\mathbf{k}} \in \Omega_{\mathbf{k}}$ chosen as $\eta$ in
Lemma \ref{series-indice-adecuado}) in
$(\cO_{X|Y_{\tau}}(s)/\cO_{X|Y_{\tau}}(<s))_p$, $p\in Y_{\tau}\cap
U_{\sigma}$, are linearly independent since the support of
$\phi_{\mathbf{k},\eta}$ restricted to the variables $x_i$ with
$i\notin \sigma$ is $\Lambda_{\mathbf{k},\eta_{\mathbf{k}}}\subseteq
\Lambda_{\mathbf{k}}$ and $\{\Lambda_{\mathbf{k}}: \; \mathbf{k}\in
\N^{n-d}\}$ is a partition of $\N^{n-d}$. This finishes the proof of
Proposition \ref{dim-gevrey-simplex}.

\section{Slopes of $\M_A (\beta )$ along coordinate hyperplanes}\label{slopes-hyperplanes}

In this section we will describe all the slopes of $\M_A (\beta )$
along coordinate hyperplanes. First, we recall here the definition
of $(A,L)$-umbrella \cite{SW}, but we will slightly modify the
notation in \cite{SW} for technical reasons. Consider any full rank
matrix $A=(a_1 \; \cdots \; a_n )\in \Z^{d\times n}$ and
$\mathbf{s}=(s_1 ,\ldots , s_n ) \in \R_{>0}^n$.

\begin{defi}
Set $a_j^{\mathbf{s}}:=a_j/s_j$, $j=1,\ldots ,n$, and let
$$\Delta_A^{\mathbf{s}}:=\operatorname{conv}(\{a_i^{\mathbf{s}} :\; i=1,\ldots ,n\}\cup \{\mathbf{0}\}) $$
be the so-called $(A, \mathbf{s})$-polyhedron.

\vspace{.3cm}

The $(A,\mathbf{s})$-umbrella is the set $\Phi_A^{\mathbf{s}}$ of
faces of $\Delta_A^{\mathbf{s}}$ which do not contain the origin.
$\Phi_A^{\mathbf{s},q}\subseteq \Phi_A^{\mathbf{s}}$ denotes the
subset of faces of dimension $q$ for $q=0,\ldots ,d-1$.
\end{defi}

The following statement is \cite[Lemma 2.13]{SW}. The difference
here is that we do not assume that $A$ is pointed but we just
consider $\mathbf{s} \in \R^{n}$ such that $s_i >0$ for all
$i=1,\ldots ,n$. For this reason we slightly modify a part of the
proof of \cite[Lemma 2.13]{SW}.

\begin{lema}\label{radinL}
Let $\widetilde{I}_A^{\mathbf{s}}$ be the ideal of $\C [\xi_1
,\ldots ,\xi_n ]$ generated by the following elements:

\begin{enumerate}
\item[i)] $\xi_{i_1} \cdots \xi_{i_r }$  where $a_{i_1}/s_{i_1} ,\ldots , a_{i_{r} }/s_{i_r}$ do not lie in a common
facet of $\Phi_A^{\mathbf{s}}$.

\item[ii)] $\xi^{u_{+}}-\xi^{u_{-}}$ where $u\in \ker_{\Z} A$ and
$\operatorname{supp}(u)$ is contained in a facet of
$\Phi_A^{\mathbf{s}}$.
\end{enumerate}

Then
$\widetilde{I}_A^{\mathbf{s}}=\sqrt{\operatorname{in}_{\mathbf{s}}
(I_A )}$.

\end{lema}

\begin{proof}
The proofs of \cite[Lemma 2.12]{SW} and \cite[Lemma 2.13]{SW} use
\cite[Lemma 2.8]{SW} and \cite[Lemma 2.10]{SW}, but they do no use
that $A$ is pointed elsewhere. We rewrite here the proofs of
\cite[Lemma 2.8]{SW} and \cite[Lemma 2.10]{SW} without the pointed
assumption on $A$ but for $\mathbf{s}\in \mathbb{R}^n_{>0}$. Let us
prove in particular that $\operatorname{in}_{\mathbf{s} } (I_A
)\subseteq \widetilde{I}_A^{\mathbf{s}} \subseteq
\sqrt{\operatorname{in}_{\mathbf{s}} (I_A )}$.

\vspace{.3cm}

For the proof of the inclusion $\operatorname{in}_{\mathbf{s}} (I_A
)\subseteq \widetilde{I}_A^{\mathbf{s} }$ we only need to prove that
$\forall u\in \Z^n $ with $Au=0$ then
$\operatorname{in}_{\mathbf{s}}
(\Box_{u})\in\widetilde{I}_A^{\mathbf{s}}$.

\vspace{.3cm}

If $\operatorname{supp}(u)\subseteq \tau \in \Phi_A^{\mathbf{s}}$
then $\exists h_{\tau} \in \mathbb{Q}^d$ such that $\langle
h_{\tau}, a_i/s_i \rangle =1, \;\forall i\in \tau$, i.e., $\langle
h_{\tau}, a_i \rangle=s_i , \;\forall i\in \tau$. Since $A u=0$ and
$\operatorname{supp}(u)\subseteq \tau$ we have $$0=\langle h_{\tau},
Au \rangle =\langle h_{\tau}A , u \rangle =\sum_{i \in \tau} s_i u_i
=\sum_{i=1}^n s_i u_i $$ so $\operatorname{in}_{\mathbf{s }}(\Box_u
)={\xi}^{u_{+} }-{\xi}^{u_{-}}$ which lies in
$\widetilde{I}_A^{\mathbf{s}}$ by definition.

\vspace{.3cm}

Assume there exists $\tau \in \Phi_A^{\mathbf{s}}$ such that
$\operatorname{supp}(u_{+})\subseteq \tau$ and
$\operatorname{supp}(u_{-})\subsetneq \tau '$ for any $\tau ' \in
\Phi_A^{\mathbf{s}}$. $h_{\tau}(a_i )=s_i \;\forall i\in \tau$ but
$h_{\tau}(a_j )< s_i \;\forall i\notin \tau$. Since $A u=0$ then
$Au_{+}=Au_{-}$ and by the assumption

$$ \sum_{i=1}^n  s_i (u_{+})_i =\langle h_{\tau}A , u_{+}\rangle=\langle h_{\tau}, A u_{+}\rangle=\langle h_{\tau}, A u_{-}\rangle <
 \sum_{i=1}^n  s_i (u_{-})_i $$ so $\operatorname{in}_{\mathbf{s}} (\Box_u )=\xi^{u_{-}}$
is a multiple of $\displaystyle \prod_{j\in \textmd{supp}(u_{-})}
\xi_j $ which is an element of the type of $i)$ by assumption.

\vspace{.3cm}

The case $\operatorname{supp}(u_{-})\subseteq \tau$ and
$\operatorname{supp}(u_{-})\subsetneq \tau '$ for any $\tau ' \in
\Phi_A^{\mathbf{s}}$ is analogous to the previous case. If there is
no face containing $\operatorname{supp}(u_{+})$ nor
$\operatorname{supp}(u_{-})$ it is trivial that $\xi^{u_{+}},
\xi^{u_{+}}\in  \widetilde{I}_A^{\mathbf{s}}$ and so $\xi^{u_{+}}-
\xi^{u_{+}}\in  \widetilde{I}_A^{\mathbf{s}}$. Finally, since $A
u_{+}=A u_{-}$ it is not possible that $\varsigma
=\operatorname{supp}(u_{+})\subseteq \tau$ and $\varsigma
'=\operatorname{supp}(u_{-}) \subseteq \tau '$ for any $\tau ,
\tau'\in \Phi_A^{\mathbf{s},d-1}$ such that $\tau \neq \tau' $
(because this implies that $\operatorname{pos}(\varsigma )
\cap\operatorname{pos}(\varsigma ') = \{\mathbf{0}\}$).

\vspace{.3cm}

Let us prove the inclusion $\widetilde{I}_A^{\mathbf{s}} \subseteq
\sqrt{\operatorname{in}_{\mathbf{s} } (I_A )}$. It is clear that the
elements of type $ii)$ lies in $\operatorname{in}_{\mathbf{s} } (I_A
)\subseteq \sqrt{\operatorname{in}_{\mathbf{s} } (I_A )}$ so we only
need to proof that the elements of type $i)$ belong to
$\sqrt{\operatorname{in}_{\mathbf{s} } (I_A )}$:

\vspace{.3cm}

If $a_{i_1}/s_{i_1} ,\ldots , a_{i_{r} }/s_{i_r}$ do not lie in a
common facet of $\Phi_A^{\mathbf{s}}$ then $\exists \mathbf{a}$ such
that:

\begin{enumerate}
\item[(1)] $\mathbf{a} \in \operatorname{conv}(a_{i_1}/s_{i_1} ,\ldots ,
a_{i_{r} }/s_{i_r})$.

\item[(2)] $\mathbf{a}$ lies in the interior of $\Delta_A^{\mathbf{s}}$.
\end{enumerate}

By $(1)$ we can write:

$$\mathbf{a}=\sum_{j=1}^r \epsilon_j a_{i_j} /s_{i_j}$$ with $\sum_{j} \epsilon_j = 1$ and $\epsilon_j \geq 0$,
$\forall j$.

\vspace{.3cm}

And by $(2)$ there exists $t>1$ such that $t a$ still belongs to
$\Delta_A^{\mathbf{s}}$ and we can write:

$$t \mathbf{a}=\sum_{i=1}^n \eta_i a_{i }/s_i $$ with $\sum_{j} \eta_j =1$ and $\eta_j \geq 0$, $\forall j$.

\vspace{.3cm}

Finally we put together both equalities and get:

$$\sum_{i=1}^r (t \epsilon_j) a_{i_j} /s_{i_j }=\sum_{i=1}^n \eta_i a_{i} / s_i .$$ Then there exists $\lambda \in \N^{\ast}$ such that
$$P=\prod_{j=1}^r \partial_{i_j}^{\lambda t \epsilon_j /s_j }
- \prod_{j=1}^n \partial_{j}^{\lambda \eta_j /s_j } \in I_A$$ and
the $\mathbf{s}$-degree of the first monomial is $\lambda t$ while
the $\mathbf{s}$-degree of the second monomial is $\lambda$ so
$\operatorname{in}_{\mathbf{s} } (P)=\prod_{j=1}^r
\partial_{i_j}^{\lambda t \epsilon_j /s_j } \in
\operatorname{in}_{\mathbf{s} } (I_A )$. This implies that $\xi_{i_1
}\cdots \xi_{i_r}\in \sqrt{\operatorname{in}_{\mathbf{s}} ( I_A )}$.

\end{proof}

Let $\tau \subseteq \{ 1,\ldots ,n \}$ be a set with cardinal $l\geq
0$ and consider the coordinate subspace $Y_{\tau}=\{ x_i =0 :\;
i\notin \tau \}$ with dimension $l$.

\vspace{.3cm}

The special filtration $$L_s :=F+(s-1) V_{\tau}$$ with $s\geq 1$ is
an intermediate filtration between the filtration $F$ by the order
of the differential operators and the Malgrange-Kashiwara filtration
with respect to $Y_{\tau}$ that we denote by $V_{\tau}$. Recall that
$V_{\tau}$ is associated with the weights $-1$ for the variables
$x_{\overline{\tau}}$, $1$ for $\partial_{\overline{\tau}}$ and $0$
for the rest of the variables.

\vspace{.3cm}

We shall identify $s\in \R_{>0}$ with $(s_1 ,\ldots ,s_n)$
throughout this section, where $s_i =1$ if $i\in \tau$ and $s_i =s$
if $i\notin \tau $. Then $(L_s )_{n+j}=s_j$ for all $j=1,\ldots ,n$.

\begin{lema}\label{bi-hom}
Assume $s > 1$ is such that
$\Phi_A^{s}=\Phi_A^{s+\epsilon}=\Phi_A^{s-\epsilon }$ for
sufficiently small $\epsilon >0$. Then the ideal $\widetilde{I}_A^{s
}$ is homogeneous with respect to $F$ and
$V_{\tau} $. In particular $\mathcal{V}(\widetilde{I}_A^{s }+\langle
A \mathbf{x} \underline{ \xi} \rangle)$ is a bi-homogeneous variety
in $\C^{2n}$.
\end{lema}

\begin{proof}
We only need to prove that the elements in Lemma \ref{radinL},
$ii)$, are bi-homogeneous. From the proof of Lemma \ref{radinL} we
deduce that they are $L_s $-homo\-geneous. By assumption we have that
they are also $(L_s \pm \epsilon V_{\tau} )$-homogeneous for all
$\epsilon
>0$ small enough.  Since $L_s\pm \epsilon V_{\tau}  =F+(s\pm \epsilon
-1)V_{\tau}$ we obtain that they are $F$-homogeneous and
$V_{\tau}$-homogeneous.
\end{proof}

\begin{lema}\label{lema-dim-n}
$\dim_{\C}(\mathcal{V}(\operatorname{in}_{L_s } (I_A ))\cap
\mathcal{V}(A \mathbf{x}\mathbf{\xi} ))\leq n$.
\end{lema}

\begin{proof}
Let $\omega\in \R^n_{>0}$ be a generic weight vector such that
$\operatorname{in}_{\omega}(\operatorname{in}_{L_s } (I_A ) )$ is a
monomial ideal. For $\epsilon >0$ small enough
$\operatorname{in}_{\omega}(\operatorname{in}_{L_s } (I_A )
)=\operatorname{in}_{\widetilde{\omega}} (I_A )$ for
$\widetilde{\omega} = s +\epsilon \omega \in \R^{n}_{>0}$.

\vspace{.3cm}

Choose any monomial order $<$ in $\C [x ,\xi ]$ that refines the
partial order given by $(u,v):=(1- \epsilon \omega_1 ,\ldots,
1-\epsilon \omega_n ; \epsilon \omega_1 ,\ldots ,\epsilon \omega_n
)\in \R^{2 n}_{>0}$. It is clear that $\operatorname{in}_{(u,v)}(A x
\xi )_i = (A x \xi)_i$ for all $i=1,\ldots ,d$ and that
$\operatorname{in}_{(u,v)}(\operatorname{in}_{L_s}(I_A))=\operatorname{in}_{\widetilde{\omega}}(I_A)$.
Then $$\operatorname{in}_{\widetilde{\omega}}(I_A)+\langle A x \xi
\rangle \subseteq
\operatorname{in}_{(u,v)}(\operatorname{in}_{L_s}(I_A) + \langle A
x\xi \rangle )$$ and so we have that:

\begin{equation}
E_{<}(\operatorname{in}_{L_s } (I_A )+\langle A
\mathbf{x}\mathbf{\xi} \rangle
)=E_{<}(\operatorname{in}_{(u,v)}(\operatorname{in}_{L_s}(I_A) +
\langle A x\xi \rangle ))\supset E_{<}
(\operatorname{in}_{\widetilde{\omega}} (I_A )+\langle A
\mathbf{x}\mathbf{\xi}\rangle) \label{escalera}
\end{equation} where $E_{<} (I):=\{ (\alpha , \gamma
)\in \N^{2n}: \; \operatorname{in}_{<}(P)= c_{\alpha , \gamma}
x^{\alpha} \xi^{\gamma}, \; P\in I \setminus \{0\} \}$ for any ideal
$I\subseteq \C [x,\xi ]$. The inclusion (\ref{escalera}) implies
that the Krull dimension of the residue ring $\C [x ,\xi
]/(\operatorname{in}_{L_s } (I_A )+\langle A \mathbf{x}\mathbf{\xi}
\rangle)$ is at most the one of $\C [x ,\xi
]/(\operatorname{in}_{\widetilde{\omega}} (I_A )+\langle A
\mathbf{x}\mathbf{\xi}\rangle )$.

\vspace{.3cm}

Then it is enough to prove that $\C [x
,\xi]/(\operatorname{in}_{\widetilde{\omega}} (I_A )+\langle A
\mathbf{x}\mathbf{\xi}\rangle )$ has Krull dimension $n$. Since
$M=\operatorname{in}_{\widetilde{\omega}} (I_A )$ is a monomial
ideal then:

$$\operatorname{in}_{\widetilde{\omega}} (I_A )=\cap_{(\partial^b ,\sigma
)\in \mathcal{S}(M)} \langle \xi_j^{b_{j}+1}:\; j\notin \sigma
\rangle$$ where $S(M)$ denotes the set of standard pairs of $M$ (see
\cite[Section 3.2]{SST}). This implies that
$$\mathcal{V}(\operatorname{in}_{\widetilde{\omega}} (I_A )+\langle
A \mathbf{x}\mathbf{\xi}\rangle)=\cup_{(\partial^b ,\sigma )\in
\mathcal{S}(M)}\mathcal{V}(\langle \xi_j :\; j\notin \sigma \rangle
+ \langle A x \xi \rangle ).$$ By  \cite[Corollary 3.2.9.]{SST}, the
columns of $A$ indexed by $\sigma$ are linearly independent when
$(\partial^b ,\sigma )\in \mathcal{S}(M)$, so the dimension of each
component
$$\mathcal{V}(\langle \xi_j :\; j\notin \sigma \rangle + \langle A x
\xi \rangle )=\mathcal{V}(\langle \xi_j :\; j\notin \sigma \rangle +
\langle x_j \xi_j : \; j\in \sigma \rangle )$$ is $n$.
\end{proof}

\begin{lema}\label{lema-descarta-slopes}
Under the assumptions of lemma \ref{bi-hom} we have that $s$ is not
an algebraic slope of $\M_{A}(\beta )$ along $Y_{\tau}$ at any point
of $Y_{\tau}$.
\end{lema}

\begin{proof}
We know that:

$$\operatorname{Ch}^s (\M_{A}(\beta ) )=\mathcal{V}(\sqrt{\operatorname{in}_{L_s } (\HAb
)})\subseteq \mathcal{V}(\sqrt{\operatorname{in}_{L_s } (I_A )})\cap
\mathcal{V}(A \mathbf{x}\mathbf{\xi} )=
\mathcal{V}(\widetilde{I}_A^{s }+\langle A \mathbf{x}\mathbf{\xi}
\rangle ).$$ Hence the $s$-characteristic variety of $\M_{A}(\beta
)$ is contained in a bi-homogeneous variety of dimension at most $n$
when the assumptions in Lemma \ref{bi-hom} are satisfied. Since
$\operatorname{Ch}^s (\M_{A}(\beta ) )$ is known to be purely
$n$-dimensional, each irreducible component is an irreducible
component of $\mathcal{V}(\widetilde{I}_A^{s }+\langle A
\mathbf{x}\mathbf{\xi} \rangle )$ and so it is also bi-homogeneous.
Moreover, this is true not only at the origin $x=0\in \R^n$ but also
at any point of $Y_{\tau}$ because $(L_s)_i= 0$ for $i\in \tau$ and
$Y_{\tau}=\{x_i =0 :\; i\notin \tau \}$. Then $s$ is not an
algebraic slope of $\M_A (\beta )$ along $Y_{\tau}$ at any point of
$Y_{\tau}$.
\end{proof}

\begin{nota}
Observe that after the proof of Lemma \ref{lema-descarta-slopes} we
have the equality in Lemma \ref{lema-dim-n}.
\end{nota}

\begin{nota}
A consequence of Lemma \ref{lema-descarta-slopes} is that $\M_A
(\beta )$ has no algebraic slopes along $\mathbf{0}\in \R^n$ at
$\mathbf{0}$.
\end{nota}

\begin{ejem}
Let $A=(a_1 \; a_2 \; a_3 \; a_4 )$ be the non-pointed matrix with
columns

$$a_1 =\left(\begin{array}{c}
        1 \\
        -1
      \end{array}\right), \; a_2 =\left(\begin{array}{c}
        0 \\
        1
      \end{array}\right), \; a_3 =\left(\begin{array}{c}
        -3 \\
        -2
      \end{array}\right), \; a_4=\left(\begin{array}{c}
        2 \\
        2
      \end{array}\right)$$ and consider the associated hypergeometric
      system: $$H_A (\beta )=I_A + \langle x_1 \partial_1 -3 x_3 \partial_3 +2 x_4 \partial_4 -\beta_1 , -x_1 \partial_1 +x_2 \partial_2 -2 x_3 \partial_3 + 2 x_4 \partial_4
      -\beta_2\rangle$$ where $I_A =\langle \partial_1 \partial_2 \partial_3 \partial_4 -1 , \partial_1 \partial_2^3 -\partial_3 \partial_4^2 , \partial_3^2 \partial_4^3 - \partial_2^2 \rangle $ and $\beta_1 ,\beta_2 \in \C$.

\vspace{.5cm}

\setlength{\unitlength}{8mm}
$$\begin{picture}(-2.5,6)(10,2)
\put(5,3){\vector(0,1){5}} \put(13,3){\vector(0,1){5}}
\put(2,5){\vector(1,0){6}} \put(10,5){\vector(1,0){6}}
\put(6,4){\line(-4,-1){4}} \put(14,4){\line(-4,-1){4}}
\put(6,4){\line(1,3){1}} \put(14,4){\line(-1,2){1}}
\put(2,3){\line(5,4){5}}\put(13,6){\line(-1,-1){3}}
\put(6,4){\makebox(0,0){$\bullet$}}\put(14,4){\makebox(0,0){$\bullet$}}
\put(5,6){\makebox(0,0){$\bullet$}}\put(13,6){\makebox(0,0){$\bullet$}}
\put(2,3){\makebox(0,0){$\bullet$}}\put(10,3){\makebox(0,0){$\bullet$}}
\put(7,7){\makebox(0,0){$\bullet$}}\put(15,7){\makebox(0,0){$\bullet$}}
\put(5,5.4){\makebox(0,0){$\bullet$}}\put(13.33,5.33){\makebox(0,0){$\bullet$}}
\put(6,4){\makebox(0.8,-0.4){$a_1$}}\put(14,4){\makebox(0.8,-0.4){$a_1$}}
\put(5,5.4){\makebox(-1.1,0.3){$\frac{2}{5} a_2
$}}\put(13,6){\makebox(0.8,0.4){$a_2$}}
\put(5,6){\makebox(0.8,0.8){$a_2$}}
\put(2,3){\makebox(0.8,-0.4){$a_3$}}\put(10,3){\makebox(0.8,-0.4){$a_3$}}
\put(7,7){\makebox(0.8,-0.4){$a_4$}}\put(15,7){\makebox(0.8,-0.4){$a_4$}}
\put(13.33,5.33){\makebox(1.2,0.1){$\frac{1}{6} a_4 $}}
\put(5,2){\makebox(0,0){Figure 2}}\put(13,2){\makebox(0,0){Figure
3}}
\end{picture}$$

From Lemma \ref{lema-descarta-slopes} we deduce that there is not
any algebraic slope along a coordinate subspace different from
$Y=\{x_2 = 0\}$ and $Z=\{x_4 =0 \}$. By Corollary
\ref{analytic-slope} and using again Lemma
\ref{lema-descarta-slopes} we know that the unique slope of $\M_A
(\beta )$ along $Y$ is $|A_{\sigma}^{-1}a_2|=5/2$ with $\sigma =\{ 3
,4\}$ and that the unique slope of $\M_A (\beta )$ along $Z$ is
$|A_{\overline{\sigma}}^{-1}a_4|=6$ with $\overline{\sigma}=\{ 1
,2\}$. Notice that $2 a_2 /5$ lies in the affine line passing
through $a_3$ and $a_4$ (see Figure 1) and that $a_4 /6$ lies in the
affine line passing through $a_1$ and $a_2$ (see Figure 2). We also
can construct $\operatorname{vol}_{\Z A } (\Delta_{\sigma})=2$
Gevrey solutions of $\M_A (\beta )$ along $Y$ (it is analogous for
$Z$) as follows.

\vspace{.3cm}

The matrix $B_{\sigma}$ is
$$B_{\sigma}=\left(\begin{array}{cc}
                                                1 & 0 \\
                                                0 & 1 \\
                                                2 & -1 \\
                                                5/2 & -3/2
                                              \end{array}\right)$$ and we consider the vectors $v^{1}=(0,0,A_{\sigma}^{-1}\beta )=(0,0,-\beta_1 +\beta_2 , -\beta_1 +
 \frac{3}{2}\beta_2 )$ and $v^2 =(0,1,A_{\sigma}^{-1}(\beta -a_2)
 )=(0,1,-\beta_1 +\beta_2 -1, -\beta_1 +
 \frac{3}{2}(\beta_2 -1 ) )$.

\vspace{.3cm}

If none of $\beta_1 -\beta_2 , -\beta_1 + \frac{3}{2}\beta_2 $ and
$-\beta_1 + \frac{3}{2}(\beta_2 -1 )$ are integers then the series
$\phi_{v^{1}}$ and $\phi_{v^{2}}$ are Gevrey series solutions along
$Y$ of $\M_A (\beta )$ with index $5/2$ at any point of $Y\cap \{
x_1 x_2 \neq 0\}$. In other case, we can replace the vectors $v^i$
by $v^{i,k}:= v^i + k (0,1,-1,-3/2 )$ with $k\in \N \setminus 2 \N$
big enough in order to obtain Gevrey solutions $\phi_{v^{i,k}} $ of
$\M_A (\beta )$ modulo convergent series at any point of $Y\cap \{
x_1 x_2 \neq 0\}$ with index $5/2$.
\end{ejem}

Denote for $s > 1$: $$\Omega_{Y_{\tau}}^{(s)} = \{ \sigma \subseteq
\tau : \; \det ( A_{\sigma })\neq 0 , \max \{ |A_{\sigma }^{-1} a_i
|:\; i\notin \tau \} = s, |A_{\sigma }^{-1} a_j |\leq 1 , \forall j
\in \tau \}.$$ Then we have the following result.

\begin{lema}\label{Lema-index-hypergeometric}
If $\sigma \in \Omega_{Y_{\tau}}^{(s_0 )}\neq \emptyset$ then for
all $p\in Y_{\tau} \cap U_{\sigma}$:

\begin{enumerate}
\item[1)] $s_0$ is the Gevrey index of a solution of $\M_A (\beta )$ in
$\cO_{\widehat{X|Y_{\tau}},p}$ for very generic parameters $\beta
\in \C^{d}$.

\item[2)] $s_0$ is the Gevrey index of a solution of $\M_A (\beta )$ in
$\cO_{\widehat{X|Y_{\tau}},p}/\cO_{X|Y_{\tau}(<s_0),p}$ for all $\beta \in
\C^d$.

\item[3)] If $Y_{\tau}$ is a hyperplane, then $s_0$ is the Gevrey index of a solution of $\M_A (\beta )$ in
$\cO_{\widehat{X|Y_{\tau}},p}/\cO_{X|Y_{\tau},p}$ for all $\beta \in\C^d$.
\end{enumerate}
\end{lema}

\begin{proof}
We consider any $\sigma \in \Omega_{Y_{\tau}}^{(s_0 )}$. If $\beta$
is very generic the Gevrey series solutions of $\M_A (\beta )$ along
$Y_{\tau}$ associated with $\sigma$, $\{
\phi_{\sigma}^{\mathbf{k}}\}_{\mathbf{k}} $ (see Section
\ref{construction-Gevrey}), have Gevrey index $s_0 =\max
\{|A_{\sigma }^{-1} a_i |:\; i\in \tau \}$ along $Y_{\tau}$ at $p\in
Y_{\tau}\cap U_{\sigma}$. If $\beta$ is not very generic we can
proceed as in Section \ref{slopes-simplex-section} in order to
construct a Gevrey series associated with $\sigma$ with index $s_0$
which is a solution of $\M_A (\beta )$ in $(\cO_{X|Y}( s_0
)/\cO_{X|Y}(<s_0))_p$ for all $p\in Y\cap U_{\sigma}$. By a similar
argument to the one in the proof of Lemma \ref{lema-slope} the
result is obtained.
\end{proof}

Assume that $Y$ is a coordinate hyperplane for the remainder of this
section. We can reorder the variables so that $Y=\{x_n = 0\}$.

\vspace{.3cm}

In the following result the equivalence of 3) and 4) is a particular
case of the Comparison Theorem of the slopes
\cite{Laurent-Mebkhout}. However, we just need to use this theorem
for the implication $3) \Longrightarrow 4)$.

\begin{teo}\label{characterization-slopes}
For all $p\in Y$ the following statements are equivalent:

\begin{enumerate}
\item[1)] $\Phi_A^{s}$ jumps at $s=s_0$.

\item[2)] $\Omega_Y^{(s_0 )} \neq \emptyset$.

\item[3)] $s_0$ is an analytic slope of $\M_A (\beta )$ along $Y$ at $p$.

\item[4)] $s_0$ is an algebraic slope of $\M_A (\beta )$ along $Y$ at $p$.
\end{enumerate}
\end{teo}

\begin{proof}
We will prove first the equivalence of 1) and 2). Assume there
exists $\sigma \in \Omega_Y^{(s_0 )} \neq \emptyset$, then
$H_{\sigma}=\{ y \in \R^d : \; |A_{\sigma}^{-1} y|=1\}$ is the only
hyperplane containing $a_i$ for all $i\in \sigma$ and
$|A_{\sigma}^{-1}(a_n /(s_0 +\epsilon ))|=s_0 /(s_0 +\epsilon)<1$,
$\forall \epsilon >0$. Hence $a_n /s_0 \in H_{\sigma}$ but $a_n
/(s_0 +\epsilon) \notin H_{\sigma}$, $\forall \epsilon >0$.

\vspace{.3cm}

Consider $\eta = \{ i: \; a_i \in H_{\sigma} \}$, then $\eta \in
\Phi_A^{s_0 + \epsilon, d-1}$, $\forall \epsilon >0$ and $n \notin
\eta$ while $\eta \cup \{n\} \in \Phi_A^{s_0 , d-1}$, so
$\Phi_A^{s}$ jumps at $s=s_0$.

\vspace{.3cm}

Conversely if $\Omega_Y^{(s_0 )} = \emptyset$ then $\forall \sigma
\subseteq \{1,2,\ldots, n-1\}$ such that $|A_{\sigma}^{-1} a_i|\leq
1$ for all $i=1,\ldots ,n-1$ we have $|A_{\sigma}^{-1} a_n|<s_0 $ or
$|A_{\sigma}^{-1} a_n|>s_0 $.

\vspace{.3cm}

Consider $\epsilon>0$ small enough such that $|A_{\sigma}^{-1}
a_n|<s_0 \pm \epsilon$ if $ |A_{\sigma}^{-1} a_n|<s_0 $ and
$|A_{\sigma}^{-1} a_n|>s_0 \pm \epsilon$ if $ |A_{\sigma}^{-1}
a_n|>s_0 $ for all simplices $\sigma$ such that $|A_{\sigma}^{-1}
a_i|\leq 1$ for all $i=1,\ldots ,n-1$.

\vspace{.3cm}

Let us prove that $\Phi_A^{s_0 ,d-1}=\Phi_A^{s_0 \pm \epsilon
,d-1}$.

\vspace{.3cm}

Assume first that $n\notin \eta\subseteq \{1 ,\ldots ,n \}$. Then:

\vspace{.3cm}

$\eta \in \Phi_A^{s_0 ,d-1} \Longleftrightarrow \exists \sigma
\subseteq \eta \mbox{ such that } | A_{\sigma}^{-1} a_i |= 1 $ for
$i\in \eta$, $| A_{\sigma}^{-1} a_i |< 1 $ for $i\notin \eta \cup \{
n\}$ and $| A_{\sigma}^{-1} a_n |< s_0 $  $\Longleftrightarrow
\exists \sigma \subseteq \eta \mbox{ such that } | A_{\sigma}^{-1}
a_i |= 1 $ for $i\in \eta$, $| A_{\sigma}^{-1} a_i |< 1 $ for
$i\notin \eta \cup \{ n\}$ and $| A_{\sigma}^{-1} a_n |< s_0 \pm
\epsilon $ $\Longleftrightarrow \eta \in \Phi_A^{s_0 \pm \epsilon
,d-1}$.

\vspace{.3cm}

If $n\in \eta \subseteq \{1 ,\ldots ,n \}$ and $\dim
(\operatorname{conv}(\eta \setminus \{ n\} ))=d-1$ then there exists
a simplex $\sigma \subseteq \eta\setminus \{ n\}$ such that $\det (
A_{\sigma } )\neq 0$. Then $\eta \notin \Phi_A^{s_0 ,d-1}$ because
in such a case $|A_{\sigma}^{-1}a_i |\leq 1$ for all $i\neq n$,
$|A_{\sigma}^{-1}a_n |=s_0 $ and so $\sigma \in \Omega_Y^{(s_0 )}$,
a contradiction. Moreover $\eta \notin \Phi_A^{s_0 \pm \epsilon
,d-1}$ for $\epsilon >0$ small enough because $|A_{\sigma}^{-1}a_n
|$ is a fixed value while $s_0 \pm \epsilon$ varies with $\epsilon$.

\vspace{.3cm}

Finally, if $n\in \eta \subseteq \{1 ,\ldots ,n \}$ and $\dim
(\operatorname{conv}(\eta \setminus \{ n\} ))<d-1$ then there exists
a hyperplane $H'=\{\mathbf{y} \in \R^d : \; h'(\mathbf{y})=0 \}$
that contains $\mathbf{0}\in \R^d$ and $a_i$ for all $i\in \eta
\setminus \{n\} $. We also can choose the linear function $h'$ in
the definition of $H'$ such that $h'(a_n )=1$. In this case:

\vspace{.3cm}

$\eta \in \Phi_A^{s_0 ,d-1}\Longleftrightarrow \eta \setminus \{n\}
\in \Phi_A^{s_0 ,d-2}$ and $\exists H'' = \{ \mathbf{y} \in \R^d :
\; h''(\mathbf{y})= 1\}$ such that $h''(a_i)=1 $ for $i\in
\eta\setminus \{n \}$, $h''(a_n )=s_0$ and $h''(a_j )<1$ for
$j\notin \eta$. This imply for $h:= h''\pm \epsilon h'$ that $h(a_i
)=1$ for all $i\in \eta \setminus \{n\}$, $h(a_n )=s_{0} \pm
\epsilon$ and $h(a_j )=h''(a_j)\pm \epsilon h'(a_j)<1$ for $j\notin
\eta$ and $\epsilon>0$ small enough because $h''(a_j )<1$ for
$j\notin \eta$. Hence $\eta \in \Phi_A^{s_0 \pm \epsilon ,d-1}$.

\vspace{.3cm}

We have proved that $\Phi_A^{s_0 ,d-1}\subseteq \Phi_A^{s_0 \pm
\epsilon ,d-1}$. This implies equality since they are
$(A,s)$-umbrellas of the same matrix $A$ and $s>0$ for $s=s_0$ and
$s=s_0 \pm \epsilon$ (in particular $\cup_{\eta \in \Phi_A^{s
,d-1}}\operatorname{pos}(\eta )=\operatorname{pos}(A)$ do not
depends on $s>0$). Moreover, the $(A,s)$-umbrellas are determined by
their facets, so $\Phi_A^{s_0 }=\Phi_A^{s_0 \pm \epsilon}$.

\vspace{.3cm}

The implication $2)\Longrightarrow 3)$ is a direct consequence of
Lemma \ref{Lema-index-hypergeometric} if $p$ belongs to the closure of $Y\cap
U_{\sigma}$ for some $\sigma \in \Omega_{Y}^{(s_0 )}$ (for example,
if $p=0$). Nevertheless, since the analytic slopes are found in
relatively open subsets of the hyperplane $Y$ we can use the
constructibility of the slopes in order to prove the result at any
point of $Y$ (see Remark \ref{nota-slope}). For the implication $3)\Longrightarrow 4)$ we use the Comparison
Theorem of the slopes \cite{Laurent-Mebkhout}. Finally, the implication $4)\Longrightarrow 1)$ is nothing but Lemma
\ref{lema-descarta-slopes}.
\end{proof}

\begin{nota}
Notice that if $Y$ is a coordinate hyperplane then every algebraic
slope $s_0$ of $\M_A (\beta)$ along $Y$  is the Gevrey index of
certain Gevrey solutions of $\M_A (\beta)$ along $Y$ modulo convergent series. Example
\ref{contraejemplo} shows that this is not true for coordinate
subspaces of codimension greater than one.
\end{nota}

\begin{ejem}\label{contraejemplo}
Let $\M_A (\beta )$ be the hypergeometric $\D$-module associated
with the matrix

$$A=\left(\begin{array}{ccc}
      1 & 0 & 3 \\
      0 & 1 & -1
    \end{array}\right)$$ and the parameter vector $\beta \in
    \C^2$. In this case $n=3=d+1$ and so the toric ideal is
    principal $I_A =\langle \partial_1^3 -\partial_2 \partial_3 \rangle$.

\vspace{.3cm}

\setlength{\unitlength}{12mm}
$$\begin{picture}(-2,5)(10,2)
\put(5,4){\vector(0,1){3}} \put(4,5){\vector(1,0){4}}
\put(5,6){\line(3,-2){3}} \put(5,5.66){\line(3,-2){2}}
\put(6,5){\makebox(0,0){$\bullet$}}
\put(5,6){\makebox(0,0){$\bullet$}}
\put(5,5.66){\makebox(0,0){$\bullet$}}
\put(8,4){\makebox(0,0){$\bullet$}}
\put(7,4.33){\makebox(0,0){$\bullet$}}
\put(6,5){\makebox(-0.5,-0.5){$a_1$}}
\put(5,6){\makebox(-0.7,0.4){$a_2$}}
\put(5,5.66){\makebox(-1,-0.3){$\frac{2}{3} a_2$}}
\put(8,4){\makebox(0.8,0.6){$a_3$}}
\put(7,4.33){\makebox(-0.6,-0.6){$\frac{2}{3} a_3$}}
\put(5,3){\makebox(0,0){Figure 4}}
\end{picture}$$

If we take $Y=\{x_2 =x_3 =0\}$ then the only algebraic slope of
$\M_A (\beta )$ along $Y$ at $p\in Y$ is $s_0 = 3/2$ (observe Figure
4 and see \cite{SW} since $A$ is pointed). Nevertheless, we will
prove that if $\beta_2 \notin \Z$ then for all $s\geq 1$,
$\mathcal{H}^0 (\operatorname{Irr}_Y^{(s)}(\M_A (\beta ))=0$:

\vspace{.3cm}

For any formal series $f=\sum_{m\in \N^2 } f_m (x_1 ) x_2^{m_2}
x_3^{m_3}$ along $Y$ at $p =(p_1 ,0 ,0 )\in Y$ then
$$(E_2 -\beta_2 )(f)=\sum_{m\in \N^2 } (m_2 - m_3 - \beta_2 ) f_m (x_1
) x_2^{m_2} x_3^{m_3} $$ and hence $(E_2 -\beta_2 )(f)\in \cO_{X,p}$
(resp. $(E_2 -\beta_2 )(f)=0$) if and only if $f \in \cO_{X,p}$
(resp. $f=0$) because $(m_2 - m_3 - \beta_2 )\neq 0$, $\forall m_2 ,
m_3 \in \N$.

\vspace{.3cm}

On the other hand, if $\beta_2 \in \Z$  we can take $k \in \N$ the
minimum natural number such that $v=(\beta_1 -3 k , \beta_2 +k , k
)\in \C \times \N^2$ has minimal negative support. Since $A v
=\beta$ then
$$\phi_v =\sum_{m\geq 0} \frac{k! [\beta_1 -3 k ]_{3
m}}{(k+m)![\beta_2 + k+m]_m } x_1^{\beta_1 -3 (k+m)}x_2^{\beta_2 +
k+m}x_3^{k+m}$$ is a formal solution of $\M_A (\beta )$ along $Y$ at
any point $p\in Y$ with $p_1 \neq 0$. In fact $\phi_v$ has Gevrey
index $s_0 = 3/2$ if $ \beta_1 - 3k \notin \N$ and it is a
polynomial when $\beta_1 - 3k \in \N$. In this last case, if we
consider $v' = v+k ' u$ with $u=(-3,1,1)\in L_A$ and $k'\in \N$ such
that $v_1 ' <0$ then $\phi_{v'} $ is a Gevrey series of index $s_0$
and $P(\phi_{v'})$ is convergent along $Y$ at any point $p\in
Y\setminus \{ 0\}$.

\vspace{.3cm}

So the algebraic slope appears as the index of a Gevrey series
solution along $Y$ of $\M_A (\beta )$ if and only if $\beta_2 \in
\Z$. Observe that "the special parameters" are not contained in a
Zariski closed set but in a countable union of them. Note also that
$I_A$ is Cohen-Macaulay and then it is known that the set of
rank-jumping parameters is empty.
\end{ejem}

\section{Regular Triangulations and
$(A,\mathbf{s})$-umbrellas}\label{triangulations-umbrella}

The aims of this section are to compare the notion of
$(A,\mathbf{s})$-umbrella in \cite{SW} with the one of regular
triangulation of the matrix $A$ (see for example \cite{Sturm}), to
show that the common domain of definition of the constructed Gevrey
series solutions $\phi_{\sigma}^{\mathbf{k}}$ is nonempty when
$\sigma$ varies in a regular triangulation and to prove the
existence of convenient regular triangulations.

\vspace{.3cm}

For any subset $\sigma \subseteq \{ 1,\ldots ,n \}$ we will write
$\operatorname{pos}(\sigma )=\sum_{i\in \sigma} \R_{\geq 0} a_i
\subseteq \R^d$. Recall that we identify $\sigma$ with $\{a_i :\:
i\in \sigma \}$.

\begin{defi}
A triangulation of $A$ is a set $\mathrm{T}$ whose elements are
subsets of columns of $A$ verifying:

\begin{enumerate}
\item[1)] $\{\operatorname{pos}(\sigma ): \; \sigma \in
\mathrm{T}\}$ is a simplicial fan.

\item[2)] $\operatorname{pos}(A)= \cup_{\sigma \in
\mathrm{T}} \operatorname{pos}(\sigma )$.
\end{enumerate}
\end{defi}

\noindent A vector $\omega \in \R^n$ defines a collection
$\mathrm{T}_{\omega}$ of subsets of columns of $A$ as follows:

\vspace{.3cm}

\noindent $\sigma \subseteq \{a_1,\ldots ,a_n \}$ belongs to
$\mathrm{T}_{\omega}$ if there exists a vector $\mathbf{c}\in \R^d$
such that

$$\langle \mathbf{c} , a_j \rangle =\omega_j  \mbox{ for all } j\in \sigma
$$ and $$\langle \mathbf{c} , a_j \rangle < \omega_j \mbox{ for all } j \notin \sigma .$$

\begin{nota}
We will say that $\omega \in \R^n$ is generic when  the collection
$\mathrm{T}_{\omega}$ is a simplicial complex and a triangulation of
$A$.
\end{nota}

\begin{defi}
A triangulation $\mathrm{T}$ is said to be regular if there exists a
generic $\omega \in \R^n$ such that $\mathrm{T} =
\mathrm{T}_{\omega}$.
\end{defi}

Observe that the collection $\{ \operatorname{pos}(\sigma ):  \;
\sigma \in \Phi_A^{\mathbf{s}} \}$ is a polyhedral fan. When
$\mathbf{s} \in \R^n_{>0}$ is generic it is a simplicial fan and so
$ \Phi_A^{\mathbf{s}}$ is a triangulation of $A$. In fact it is a
regular triangulation because for any $\mathbf{s}\in \R^n_{>0}$ we
have that $\Phi_{A}^{\mathbf{s}}=\mathrm{T}_{\mathbf{s}}$:

\vspace{.3cm}

$\sigma \in \Phi_{A}^{\mathbf{s}}\Longleftrightarrow \exists
\mathbf{c}\in \R^d | \; \langle \mathbf{c} , a_i /s_i \rangle=1, \;
\forall i\in \sigma$, and $\langle \mathbf{c} , a_i /s_i \rangle <1,
\; \forall i\notin \sigma \Longleftrightarrow \exists \mathbf{c}|
\;\langle \mathbf{c} , a_i \rangle=s_i , \; \forall i\in \sigma$ and
$\langle \mathbf{c} , a_i \rangle < s_i , \; \forall i\notin \sigma
\Longleftrightarrow \sigma \in \mathrm{T}_{\mathbf{s}} $.

\vspace{.3cm}

Given a ($d-1$)-simplex $\sigma \in \mathrm{T}_{\omega }$ there
exists $\mathbf{c}\in \R^d$ such that
$\mathbf{c}A_{\sigma}=\omega_{\sigma}$ and
$\mathbf{c}A_{\overline{\sigma}}<\omega_{\overline{\sigma}}$. This
is equivalent to:

$$\mathbf{c}= \omega_{\sigma} A_{\sigma}^{-1},\;
\omega_{\sigma}A_{\sigma}^{-1}A_{\overline{\sigma}} <
\omega_{\overline{\sigma}}.$$ But this happens if and only if
$\omega \in C(\sigma ):=\{ \omega \in \R^n :\; \omega B_{\sigma }>
0\}$ which is an open convex polyhedral rational cone of dimension
$n$. Then we can write
$$C(\sigma )=\{\omega \in \R^n : \; \sigma \in
\mathrm{T}_{\omega}\}$$ and for any regular triangulation
$\mathrm{T} = \mathrm{T}_{\omega_0 }$ we have
$$\omega_0 \in C(\mathrm{T}):= \bigcap_{\sigma \in \mathrm{T}} C(\sigma ).$$ Hence
$C(\mathrm{T})= \{ \omega \in \R^{n} :\;
\mathrm{T}_{\omega}=\mathrm{T} \}$ is a nonempty open rational
convex polyhedral cone. It is clear that $\cup_{\mathrm{T}}
\overline{C(\mathrm{T})}=\R^{n}$ where $\mathrm{T}$ runs over all
regular triangulations of $A$ and $\overline{C(\mathrm{T})}$ denotes
the Euclidean closure of $C(\mathrm{T})$. More precisely, there
exists a polyhedral fan with support $\R^{n}$ such that
$\mathrm{T}_{\omega}$ is constant for $\omega\in \R^{n}$ running in
any relatively open cone of this polyhedral fan. We also can
restrict this fan to $\R^{n}_{>0}$ and obtain that the
$(A,\mathbf{s})$-umbrella is constant for $\mathbf{s}\in
\R^{n}_{>0}$ running in any relatively open cone.

\vspace{.3cm}

Recall from (\ref{domain}) that $$U_{\sigma} =\{x\in \C^n : \;
\prod_{i\in \sigma} x_{i} \neq 0 , \; (-\log |x_1 |,\ldots ,-\log
|x_n |)B_{\sigma , j } > -\log R, \; \forall a_j \in
H_{\sigma}\setminus \sigma \}$$ where $B_{\sigma , j }$ is the
$j$-th column of $B_{\sigma }$, i.e. the vector with
$\sigma$-coordinates $-A_{\sigma}^{-1}a_j$ and
$\overline{\sigma}$-coordinates equal to the $j$-th column of the
identity matrix of order $n-d$. Then $U_{\sigma}$ contains those
points $x\in \C^n \cap \{\prod_{i\in \sigma } x_i \neq 0\}$ for
which $$(-\log |x_1 |,\ldots ,-\log |x_n |)$$ lies in a sufficiently
far translation of the cone $C(\sigma )$ inside itself. Then for any
regular triangulation $\mathrm{T}$ of $A$ we have that $\cap_{\sigma
\in \mathrm{T} } U_{\sigma} $ is a nonempty open set since it
contains those points $x\in \C^n \cap \{ \prod_{i\in \sigma } x_i
\neq 0 :\; \sigma \in \mathrm{T}\}$ for which $(-\log |x_1 |,\ldots
,-\log |x_n |)\in \R^n$ lies in a sufficiently far translation of
the nonempty open cone $C( \mathrm{T}  )$ inside itself.

\vspace{.3cm}

\begin{lema}\label{lema-existence-triangulation}
Given a full rank matrix $A\in \Z^{d\times n}$ with $d\leq n$ and a
lattice $\Lambda$ with $A\subseteq \Lambda  \subseteq \Z^d$ there
exists a regular triangulation $\mathrm{T}$ of $A$ such that
\begin{equation}
\operatorname{vol}_{\Lambda}(\Delta_{A})=\sum_{\sigma \in
\mathrm{T}, \dim \sigma = d-1} \operatorname{vol}_{\Lambda}
(\Delta_{\sigma}) \label{condition-triangulation}
\end{equation}
\end{lema}

\begin{proof}
The volume function with respect to a lattice $\Lambda$ is nothing
but the Euclidean volume function normalized so that the unit
simplex in $\Lambda$ has volume one. Hence, it is enough to proof
the result for the Euclidean volume.

\vspace{.3cm}

If $A$ is such that the facets of $\Delta_A$ contain exactly $d$
columns of $A$, then they are simplicial facets and so the regular
triangulation $T_{\omega}$ with $\omega =(\omega_1, \ldots
,\omega_n)$ and $\omega_i =1$, $\forall i$,  verifies the desired
condition.

\vspace{.3cm}

Assume now that there exists $\tau \in T_{\omega}$ with at least
cardinal $d+1$. Then we can take $i \in \tau$ such that the columns
of $A$ in $\tau \setminus \{i \}$ determines a hyperplane
$H_{\tau}$. Consider all the hyperplanes $H_{\tau'}\neq H_{\tau}$
determined by facets $\tau'$ of $\Delta_A$ not containing neither
the origin nor $a_i$. Then $a_i$ lies in the interior of the
polytope $\cap_{\tau '} (H_{\tau '} \cup H_{\tau '}^{-})$ and so
$\frac{a_i}{1-\epsilon}$ does too for $\epsilon
>0$ small enough. This means that we do not modify the facets that
not contain $a_i$ via replacing it by  $\frac{a_i}{1-\epsilon}$. The
facets $\tau'$ containing $a_i$ such that $\tau '\setminus \{a_i \}$
does not determine $H_{\tau '}$ are not modified neither. We just
modify the facets $\tau'$ of $\Delta_A$ that contain $a_i$ and $\tau
' \setminus \{a_i \}$ determine $H_{\tau '}$. Such a kind of facet
is replaced by more than one facet with vertices contained in
$\tau'$ and hence each of the new facets contain less columns of $A$
than the original one. This process finishes in a finite number of
steps and yields to a polytope with simplicial facets. We have
replaced each $a_i$ by $a_i /\omega_i$ with $\omega_i \in \R_{>0}$
and this is equivalent to consider $T_{\omega}$ with
$\omega=(\omega_1 ,\ldots ,\omega_n )\in \R_{>0}^n$. Moreover, we do
not modify the set of vertices and thus $T_{\omega}$ satisfies
(\ref{condition-triangulation}).
\end{proof}

\section{Gevrey solutions of $\M_A (\beta )$ along coordinate subspaces}

\subsection{Lower bound for the dimension}\label{basis-Gevrey}

\indent In this section we provide an optimal lower bound in terms
of volumes of polytopes of the dimension of $\mathcal{H}om_{\D}
(\M_A (\beta ), \cO_{X|Y_{\tau}}(s))_p$, $s\in \R$, for generic
points $p\in Y_{\tau}=\{x_i =0:\; i\notin \tau \}$ and for all
$\beta \in \C^d$.

\vspace{.3cm}

Consider the submatrix $A_{\tau}=(a_i )_{i\in \tau}$ of $A$. If the
rank of $A_{\tau}$ is $d$ then there exists a regular triangulation
$\mathrm{T}(\tau )$ of $A_{\tau}$ such that

\begin{equation}
\operatorname{vol}_{\Z A}(\Delta_{\tau})=\sum_{\sigma \in
\mathrm{T}(\tau ), \dim \sigma =d-1} \operatorname{vol}_{\Z
A}(\Delta_{\sigma}) \label{condition-triangulation-2}
\end{equation} because of Lemma \ref{lema-existence-triangulation}.
If the rank of $A_{\tau}$ is lower than $d$ then this equality holds
for any regular triangulation of the matrix $A_{\tau}$ since all the
volumes in (\ref{condition-triangulation-2}) are zero.

\vspace{.3cm}

In this section we shall identify $s\in \R_{>0}$ with $(s_1 ,\ldots
,s_n )$ such that $s_i =1$ for $i\in \tau$ and $s_i =s$ for $i\notin
\tau$. Since $A_{\tau }$ is a submatrix of $A$ there exists a
regular triangulation $\mathrm{T}$ of $A$ such that $\mathrm{T}(\tau
)\subseteq \mathrm{T}$.

\vspace{.3cm}

For all $s\in \R$ we consider the following subset of
$\mathrm{T}(\tau )$:

$$\mathrm{T}(\tau , s):=\{\sigma \in
\mathrm{T}(\tau ): \; \dim (\sigma )=d-1, \; a_j /s \notin
H_{\sigma}^{+} \; \forall j\notin \tau \}.$$

The following theorem is the main result in this section.

\begin{teo}\label{GevreyII}
For all $\tau\subseteq \{1,\ldots ,n\}$,
\begin{equation}
\dim_{\C }\mathcal{H}om_{\D} (\M_A (\beta ), \cO_{\widehat{X|Y_{\tau}}})_p
\geq \operatorname{vol}_{\Z A }(\Delta_{\tau})\label{lower-bound}
\end{equation} for $p$ in the nonempty relatively open set $W_{\mathrm{T}(\tau )}
:=Y_{\tau}\cap (\bigcap_{\sigma \in \mathrm{T}(\tau )}U_{\sigma})$.
More precisely,

\begin{equation}
\dim_{\C }\mathcal{H}om_{\D} (\M_A (\beta ), \cO_{X|Y_{\tau}}(s))_p
\geq \sum_{\sigma \in \mathrm{T}(\tau , s)} \operatorname{vol}_{\Z
A}(\Delta_{\sigma} ) \label{lower-bound-s}
\end{equation} for all $s\in \R$ and $p$ in the nonempty
relatively open set $W_{\mathrm{T}(\tau , s)} :=Y_{\tau} \cap
(\bigcap_{\sigma \in \mathrm{T}(\tau , s)} U_{\sigma})$.
\end{teo}

\begin{proof}
$W_{\mathrm{T}(\tau )} \subseteq W_{\mathrm{T}(\tau , s)}$ are
nonempty relatively open subsets of $Y_{\tau}$ because
$\mathrm{T}(\tau )$ is a regular triangulation of $A_{\tau}$ (see
Section \ref{triangulations-umbrella}).

\vspace{.3cm}

For each fixed $(d-1)$-simplex $\sigma \in \mathrm{T}(\tau ,s)$, we
have that $|A_{\sigma}^{-1}a_j |\leq 1 $ for all $j\in \tau$ and
$|A_{\sigma}^{-1}a_j |\leq s $ for all $j\notin \tau$ and we can
construct $\operatorname{vol}_{\Z A}(\Delta_{\sigma} )$ Gevrey
solutions of $\M_{A}(\beta )$ of order $s$ along $Y_{\tau}$ at any
point of $Y_{\tau }\cap U_{\sigma}$ (see Section
\ref{construction-Gevrey}). These $\operatorname{vol}_{\Z
A}(\Delta_{\sigma} )$ series $\{ \phi_{\sigma}^{\mathbf{k}}
\}_{\mathbf{k}}$ are linearly independent because they have pairwise
disjoint supports. The linear independency of the set of all $\operatorname{vol}_{\Z A }(\Delta_{\tau})$ series $\phi_{\sigma}^{\mathbf{k}}$
when $\sigma$ varies in $\mathrm{T}(\tau )$ is also clear if we assume that $\beta$ is very generic (because this
implies that they have pairwise disjoint supports).

\vspace{.3cm}

If $\beta$ is not very generic some of the series could be equal up
to multiplication by a nonzero scalar. In such a case one can
proceed similarly to the proof of Theorem 3.5.1. in \cite{SST}:

\vspace{.3cm}

We introduce a perturbation $\beta \mapsto \beta +\epsilon \beta '$
with $\beta '\in \C^d$ such that $\beta +\epsilon \beta '$ is very
generic for $\epsilon \in \C$ with $|\epsilon|>0$ small enough (it
is enough to consider $\beta '\in \C^d$ such that $
(A_{\sigma}^{-1}\beta')_i \neq 0$ for $i=1,\ldots ,d$ and $\sigma
\in \mathrm{T}(\tau )$).

\vspace{.3cm}

Consider the set $\{\phi_{\sigma}^{\mathbf{k}}: \; \sigma \in
\mathrm{T}(\tau ) , \mathbf{k} \in \N^{n-d}\}$ with
$\operatorname{vol}_{\Z A}(\Delta_{\tau})$ Gevrey series solutions
of $\M_A (\beta + \epsilon \beta ')$ with disjoint supports. We will
denote these series by $\phi_{\sigma}^{\mathbf{k}}(\beta + \epsilon
\beta ')$ in this proof. It is clear that
$\phi_{\sigma}^{\mathbf{k}}(\beta + \epsilon \beta
')=\phi_{v_{\sigma}^{\mathbf{k}}(\beta + \epsilon \beta ')}$ for
$$v_{\sigma}^{\mathbf{k}}(\beta + \epsilon \beta ')=
v_{\sigma}^{\mathbf{k}}(\beta)+\epsilon
v_{\sigma}^{\mathbf{0}}(\beta ').$$ Here
$v_{\sigma}^{\mathbf{k}}(\beta )$ has $\sigma$-coordinates
$A_{\sigma}^{-1}(\beta - A_{\overline{\sigma}} \mathbf{k})$ and
$\overline{\sigma}$-coordinates $\mathbf{k}$. Similarly,
$v_{\sigma}^{\mathbf{0}}(\beta ')$ has $\sigma$-coordinates
$A_{\sigma}^{-1} \beta '$ and $\overline{\sigma}$-coordinates
$\mathbf{0}$. Let $\mathrm{T}$ be a regular triangulation of $A$
such that $\mathrm{T} (\tau )\subseteq \mathrm{T}$. For any
$\phi_{\sigma}^{\mathbf{k}}(\beta )$ we can assume without loss of
generality that $v_{\sigma}^{\mathbf{k}}(\beta )$ has minimal
negative support, $\phi_{\sigma}^{\mathbf{k}}(\beta
)=\phi_{v_{\sigma}^{\mathbf{k}}(\beta )}$ and
$\operatorname{in}_{\omega}(\phi_{\sigma}^{\mathbf{k}}(\beta ))=
x^{v_{\sigma}^{\mathbf{k}}(\beta )}$ for a fixed generic $\omega \in
\mathcal{C}(\mathrm{T})$. Then for two simplices $\sigma , \sigma '
\in \mathrm{T} (\tau )$ we have that
$\phi_{v_{\sigma}^{\mathbf{k}}(\beta )}=\phi_{v_{\sigma
'}^{\mathbf{k}'}(\beta )}$ if and only if
$v_{\sigma}^{\mathbf{k}}(\beta)=v_{\sigma ' }^{\mathbf{k} '
}(\beta)$.

\vspace{.3cm}

Let us denote $\nu = \operatorname{vol}_{\Z A}(\Delta_{\tau})$ and
consider the $\nu$ linearly independent Gevrey series solutions of
$\M_A (\beta)$ along $Y_{\tau}$ of the form

$$\phi_{\sigma}^{\mathbf{k}}(\beta + \epsilon \beta ')=
\sum_{\mathbf{k}+\mathbf{m} \in \Lambda_{\mathbf{k}}}
q_{k+m}(\epsilon) x^{v_{\sigma}^{\mathbf{k}}(\beta + \epsilon \beta
')+u(\mathbf{m})}$$ where
$$q_{k+m}(\epsilon)=\frac{[v_{\sigma}^{\mathbf{k}}(\beta)+ \epsilon
v_{\sigma}^{\mathbf{0}} (\beta
')]_{u(\mathbf{m})_{-}}}{[v_{\sigma}^{\mathbf{k}}(\beta)+ \epsilon
v_{\sigma}^{\mathbf{0}} (\beta ')+ u(m)]_{u(\mathbf{m})_{-}}}$$ for
$\sigma \in \mathrm{T}(\tau )$ and $\mathbf{k}\in \N^{n-d}$
verifying that $\phi_{\sigma}^{\mathbf{k}}(\beta )=
\phi_{v_{\sigma}^{\mathbf{k}}}(\beta )$. Observe that for all
$\mathbf{k}+\mathbf{m} \in \Lambda_{\mathbf{k}}$ we can write
$$x^{v_{\sigma}^{\mathbf{k}}(\beta)+ \epsilon v_{\sigma}^{\mathbf{0}}
(\beta ')+ u(m)}=e^{\epsilon \log x_{\sigma}^{A_{\sigma}^{-1} \beta
'}} x^{v_{\sigma}^{\mathbf{k}}(\beta)}.$$ Then we have:

$$\phi_{\sigma}^{\mathbf{k}}(\beta + \epsilon \beta ')=
e^{\epsilon \log x_{\sigma}^{A_{\sigma}^{-1} \beta '}}
\sum_{\mathbf{k}+\mathbf{m} \in \Lambda_{\mathbf{k}}}
q_{k+m}(\epsilon) x^{v_{\sigma}^{\mathbf{k}}(\beta
)+u(\mathbf{m})}.$$ It is clear that
$q_{\mathbf{k}+\mathbf{m}}(\epsilon)$ is a rational function on
$\epsilon$ and it has a pole of order $\mu_{\mathbf{k}+\mathbf{m}}$
with $0\leq \mu_{\mathbf{k}+\mathbf{m}} \leq d$. On the other hand
$e^{\epsilon \log x_{\sigma}^{A_{\sigma}^{-1} \beta '}}=\sum_{l\geq
0} \frac{(\log (x_{\sigma}^{A_{\sigma}^{-1} (\beta ')}))^l
}{l!}\epsilon^l$ so we can expand the series $\epsilon^{\mu}
\phi_{\sigma}^{\mathbf{k}}(\beta +\epsilon \beta ')$ (with $\mu =
\max \{\mu_{\mathbf{k}+\mathbf{m}}\}\leq d$) and write it in the
form $\sum_{j\geq 0} \phi_{j} (x) \epsilon^j$ where $\phi_{0}
(x)\neq 0$ and $\phi_{j} (x) $ are Gevrey solutions of $\M_A (\beta
)$ along $Y_{\tau}$ that converge in a common relatively open subset
of $Y_{\tau}$ for all $j$.

\vspace{.3cm}

After a reiterative process making convenient linear combinations of
the series and dividing by convenient powers of $\epsilon$, one
obtain $\nu $ Gevrey solutions of $\M_A (\beta +\epsilon \beta ')$
of the form $\sum_{j\geq 0} \psi_{i,j} (x) \epsilon^j$ where
$\psi_{i,0} (x)\neq 0$, $i=1,\ldots ,\nu$, are linearly independent.
Then we can substitute $\epsilon = 0$ and obtain the desired $\nu$
linearly independent Gevrey series solutions of $\M_A (\beta )$. The
logarithms $\log (x_i )$ just appear for $i\in \sigma$ with $\sigma
$ varying in $\mathrm{T} (\tau )$ at any step of the process. Thus
the $\nu =\operatorname{vol}_{\Z A }(\Delta_{\tau})$ final series
just have logarithms $\log (x_i )$ with $i\in \tau$. Hence these
series are Gevrey series solutions of $\M_A (\beta )$ along
$Y_{\tau}$ at points of $W_{\mathrm{T}(\tau )}$. Moreover, it is
clear that the Gevrey index cannot increase with this process.
\end{proof}

\begin{nota}\label{nota-DM}
The proof of Proposition 5.2. in \cite{Saito} guarantees that all
the series solutions obtained after the process that we mention in
the proof of Theorem \ref{GevreyII} have the form
$$\sum_{v} g_v (log (x_i ): \; i\in \tau) x^v$$ with $g_v (y_{\tau })$ a
polynomial in $\C [ y_{\tau}^u : \; u\in L_{A_{\tau}} ]$.
\end{nota}

\begin{nota}
Theorem \ref{GevreyII} generalizes \cite[Theorem 3.5.1]{SST} and
\cite[Corollary 1]{T} (taking $\tau=\{1,\ldots ,n\} $ and $s=1$ in
(\ref{lower-bound-s})), that establish that the holonomic rank of a
hypergeometric system (i.e. the dimension of the space of
holomorphic solutions at nonsingular points) is greater than or
equal to $\operatorname{vol}_{\Z A}(\Delta_A )$. A more precise
statement than \cite[Corollary 1]{T} is given in \cite{MMW}: the
holonomic rank is upper semi-continuous in $\beta$ for holonomic
families, including hypergeometric systems $\M_A (\beta )$ with $A$
a pointed matrix.
\end{nota}

\begin{nota}
Different regular triangulations $\mathrm{T}(\tau )$ of $A_{\tau}$
verifying the condition (\ref{condition-triangulation-2}) will
produce different sets with $\operatorname{vol}_{\Z A} (\Delta_{\tau
})$ linearly independent solutions of $\M_A (\beta )$ in
$\cO_{\widehat{X|Y_{\tau }},p}$ for $p$ in pairwise disjoint open
subsets $W_{\mathbf{T(\tau )}}$ of $Y_{\tau}$. It is natural to ask
whether $\cup_{\mathbf{T (\tau )}}\overline{W_{\mathbf{T(\tau
)}}}=Y_{\tau}$ for $\mathbf{T (\tau )}$ running over all possible
regular triangulations $\mathbf{T(\tau )}$ of $A_{\tau }$ verifying
(\ref{condition-triangulation}). We have that $\cup_{\mathbf{T (\tau
)}}\overline{C(\mathbf{T (\tau )})}=\R^{l}$ and that there exists
$w,w'\in C(\mathbf{T (\tau )})$ verifying two-sided Abel lemma:
$$w+C(\mathbf{T (\tau )})\subseteq - \operatorname{Log}
W_{\mathrm{T}(\tau )}\subseteq w'+C(\mathbf{T(\tau )})$$ where
$\operatorname{Log}: \C^{l} \longrightarrow \R^{l}$ is the map
$\operatorname{Log}(x_1 ,\ldots ,x_l )=(\log |x_1 | ,\ldots ,\log
|x_l |)$. This should be contrasted with \cite[Lemma 11]{PST}.
\end{nota}

\begin{nota}
If there are no more than $d$ columns of $A_{\tau}$ in the same
facet of $\Delta_{\tau}$ then by Theorem \ref{GevreyI} all the
series above are Gevrey of the corresponding order along $Y_{\tau}$
at any point of $Y_{\tau}\cap (\cap_{\sigma \in \mathrm{T}(\tau )}
\{ \prod_{i\in \sigma } x_{i}\neq 0\})$.
\end{nota}

\begin{nota}
An anonymous referee of the paper \cite{FC2} made us the following
question. Is there some understanding how Gevrey solutions of
$\M_A(\beta)$ relate to solutions of $\M_{A^h}(\beta^h)$ with $A^h$
the matrix obtained from A by adding a row of 1's and then a column
equal to the first unit vector? The idea is to consider a regular
triangulation $\mathrm{T}$ for the matrix $A^h$ containing a regular
triangulation $\mathrm{T} (\tau)$ of $A_{\tau}^h$. For any simplex
$\sigma \in T ( \tau )$, the dehomogenization (in the sense of
\cite[Definition 2]{Ohara-Takayama}) of the holomorphic solutions
$\phi_{\sigma}^{\mathbf{k}}$ of $M_{A^h}(\beta^h)$ are Gevrey
solutions of $M_A(\beta)$ with respect to $Y_\tau$. We will give
more details about this subject in a forthcoming paper.
\end{nota}

\subsection{Dimension for very generic parameters}\label{basis-Gevrey2}

Let $\tau\subseteq \{ 1, \ldots , n\}$ be a subset with cardinal
$l$, $1\leq l\leq n-1$, and recall that we denote $Y_{\tau}=\{x_i
=0:\; i\notin \tau \}$.

\vspace{.5cm}

The aim of this section is to prove the following result:

\begin{teo}\label{theorem-dimension-formal}
For generic $p\in Y_{\tau}$ and very generic $\beta$,
$$\dim_{\C}\mathcal{H}om (\M_A (\beta ),\cO_{\widehat{X|Y_{\tau}}})_p
=\operatorname{vol}_{\Z A} (\Delta_{\tau}).$$
\end{teo}

\begin{nota}
Theorem \ref{theorem-dimension-formal} also implies that equality
holds in (\ref{lower-bound-s}) for very generic parameters $\beta
\in \C^d$ because the $\operatorname{vol}_{\Z A} (\Delta_{\tau})$
Gevrey series $\phi_{\sigma}^{\mathbf{k}}$ with $\sigma \in
\mathrm{T} (\tau )$ have pairwise disjoint supports and their index
along $Y_{\tau}$ is $\max \{|A_{\sigma}^{-1} a_j |:\; j\notin \tau
\}$.
\end{nota}

\begin{cor}
If $\beta \in \C^d$ is very generic then
\begin{equation}\dim_{\C } \mathcal{H}^0(\operatorname{Irr}_{Y_{\tau}}^{(s)} (\M_{A}(\beta )
)_p \geq \sum_{\sigma \in \mathrm{T}(\tau , s)
\setminus\mathrm{T}(\tau , 1)} \operatorname{vol}_{\Z
A}(\Delta_{\sigma} ) \label{lower-bound-irreg-s}
\end{equation} for generic $p\in Y_{\tau}$.
\end{cor}

In Section \ref{basis-Gevrey} we proved the lower bound
(\ref{lower-bound}) by explicitly constructing
$\operatorname{vol}_{\Z A} (\Delta_{\tau})$ Gevrey series solutions
of $\M_A (\beta )$ along $Y_{\tau}$ in certain relatively open
subsets of $Y_{\tau}$. Now we are going to prove that equality holds
if $\beta$ is very generic.

\begin{lema}
If $\operatorname{rank}(A_{\tau })=d$ then $$\operatorname{vol}_{\Z
A}(\Delta_{\tau})= \operatorname{vol}_{\Z \tau }(\Delta_{\tau}) [\Z
A : \Z \tau ].$$
\end{lema}

\begin{proof}
We have that $\operatorname{vol}_{\Z A}(\Delta_{\tau})=\frac{d!
\operatorname{vol}(\Delta_{\tau})}{[\Z^d :\Z A]}$ and
$\operatorname{vol}_{\Z \tau}(\Delta_{\tau})=\frac{d!
\operatorname{vol}(\Delta_{\tau})}{[\Z^d :\Z \tau]}$. Since $\Z \tau
\subseteq \Z A\subseteq \Z^d $ then $[\Z^d :\Z \tau ]= [\Z^d :\Z A] [\Z A : \Z \tau ]$
and the result is obtained.
\end{proof}

\begin{lema}\label{lema-formal-solutions}
If $f=\sum_{m\in \N^{n-l}} f_m (x_{\tau}) x_{\overline{\tau}}^{m}
\in \cO_{\widehat{X|Y_{\tau}},p}$ is a formal solution of $\M_A
(\beta )$, then $f_m (x_{\tau}) \in \cO_{Y_{\tau},p}$ is a
holomorphic solution of $\M_{A_{\tau}}(\beta -A_{\overline{\tau}}m)$
for all $m\in \N^{n-l}$.
\end{lema}

\begin{proof}
It is clear that $I_A \cap \C [\partial_{\tau}]=I_{A_{\tau}}$. Then
for any differential operators $P\in I_{A_{\tau}}\subseteq \C
[\partial_{\tau}]$ we have that $$0=P(f)=\sum_{m\in \N^{n-l}} P(f_m
(x_{\tau})) x_{\overline{\tau}}^{m}$$ and this implies that $ P(f_m
(x_{\tau}))=0$ for all $m\in \N^{n-l}$.

\vspace{.5cm}

Let $\Theta$ denote the vector with coordinates $\Theta_i =x_i
\partial _i$ for $i=1,\ldots ,n$. Then  $A\Theta -\beta =A_{\tau}\Theta_{\tau} +
A_{\overline{\tau}}\Theta_{\overline{\tau}}-\beta$ and
$$\mathbf{0}=(A\Theta -\beta )(f)=\sum_{m\in \N^{n-l}}(A_{\tau}\Theta_{\tau} +
A_{\overline{\tau}}m -\beta )( f_m (x_{\tau}
))x_{\overline{\tau}}^{m}$$ so $f_m (x_{\tau} )$ must be annihilated
by the Euler operators $A_{\tau} \Theta_{\tau} -(\beta
-A_{\overline{\tau}} m)$.
\end{proof}

\begin{cor}\label{coro-zero-solution}
If $\operatorname{rank}(A_{\tau })<d$ and $\beta \in \C^d$ is very
generic then $$dim_{\C}\mathcal{H}om (\M_A (\beta
),\cO_{\widehat{X|Y_{\tau}}})= 0 .$$
\end{cor}

\begin{proof}
If $\operatorname{rank}(A_{\tau })<d$, then there exists a nonzero
vector $\gamma \in \Q^d$ such that the vector $\gamma A_{\tau}$ is
zero. If $\beta$ is very generic $(\gamma A_\tau \Theta_{\tau}
-\gamma (\beta - A_{\overline{\tau }}m) = -\gamma (\beta -
A_{\overline{\tau }}m)\neq 0$ is a nonzero constant that is a linear
combination of the Euler operators in the definition of
$\M_{A_{\tau}}(\beta - A_{\overline{\tau }}m)$ and so
$\M_{A_{\tau}}(\beta - A_{\overline{\tau }}m)=0$. By Lemma
\ref{lema-formal-solutions}, the coefficients in $\cO_{Y_{\tau} ,p}$
of any formal solution $f$ of $\M_A (\beta )$ in
$\cO_{\widehat{X|Y_{\tau}},p}$ must be solutions of
$\M_{A_{\tau}}(\beta - A_{\overline{\tau }}m)=0$. This implies that
the coefficients of $f$ are zero and so $f=0$.
\end{proof}

\begin{nota}
By Corollary \ref{coro-zero-solution} we have the equality in
Theorem \ref{theorem-dimension-formal} holds when
$\operatorname{rank}(A_{\tau })<d$. For the remainder of this
section we shall assume that $\operatorname{rank}(A_{\tau })=d$ and
then $l\geq d$.
\end{nota}

The following Lemma is a direct consequence of results from
\cite{Adolphson} and \cite{GKZ}.

\begin{lema}\label{lema-holomorphic}
If $\beta$ is very generic and $p\in Y_{\tau }$, then for all $m\in
\N^{n-l}$:
$$\dim_{\C}\mathcal{H}om (\M_{A_{\tau }} (\beta -A_{\overline{\tau}} m ),
\cO_{Y_{\tau}})_p \leq \operatorname{vol}_{\Z \tau}
(\Delta_{\tau}).$$ Equality holds if $p$ does not lie in the
singular locus of $\M_{A_{\tau }} (\beta )$ (which does not depend
on $\beta$).
\end{lema}

Let us consider $\mathrm{T} (\tau)$ a regular triangulation of
$A_{\tau}$ verifying (\ref{condition-triangulation-2}).

\begin{lema}\label{lema-formal}
Any formal solution $f=\sum_{m\in \N^{n-l}} f_m (x_{\tau})
x_{\overline{\tau}}^{m} \in \cO_{\widehat{X|Y_{\tau}},p}$ of $\M_A
(\beta )$, $p\in W_{\mathrm{T}(\tau )}\subseteq Y_{\tau}$, can be
written as follows:

$$f=\sum_{\sigma \in \mathrm{T}(\tau )}\sum_{\mathbf{m}\in \N^{n-d}} c_{\sigma ,
\mathbf{m}} x_{\sigma}^{A_{\sigma}^{-1}(\beta
-A_{\overline{\sigma}}\mathbf{m})}x_{\overline{\sigma}}^{\mathbf{m}}.$$

\end{lema}

\begin{proof}
By Lemma \ref{lema-holomorphic} a basis of $\mathcal{H}om
(\M_{A_{\tau}}(\beta -A_{\overline{\tau}}\mathbf{m}_{\overline{\tau}},\cO_{Y_{\tau},p})$
for $p\in W_{\mathrm{T}(\tau )}\subseteq Y_{\tau }$ is given by the $\operatorname{vol}_{\Z \tau}(\Delta_{\tau})$ series
$\phi_{\sigma}^{\mathbf{k}}$ with $\sigma$ running in the
($d-1$)-simplices of $\mathrm{T}(\tau )$ and $\Lambda_{\mathbf{k}}$
running in the partition of $\N^{l-d}$ (see Remark \ref{partition}
and apply it to the matrix $A_{\tau}$ with $l$ columns and $\sigma
\subseteq \tau$). In particular we obtain that:

$$f_{m_{\overline{\tau}}}(x_{\tau})=\sum_{\sigma \in \mathrm{T}(\tau
)}\sum_{m_{\overline{\sigma} \cap \tau}\in \N^{l-d}} c_{\sigma ,
\mathbf{m}_{\overline{\sigma}}}x_{\sigma}^{A_{\sigma}^{-1}(\beta
-A_{\overline{\tau}}\mathbf{m}_{\overline{\tau}}-A_{\overline{\sigma}
\cap \tau}\mathbf{m}_{\overline{\sigma} \cap
\tau})}x_{\overline{\sigma} \cap
\tau}^{\mathbf{m}_{\overline{\sigma} \cap \tau}}$$ and this implies
the result.
\end{proof}

Using the partition $\{\Lambda_{\mathbf{k}(i)}:\; i=1,\ldots ,r\}$
of $\N^{n-d}$ (see Remark \ref{partition}) with $r=[\Z A :\Z \sigma
]$ we can write the formal solution in Lemma \ref{lema-formal} as:

$$f=\sum_{\sigma \in \mathrm{T}(\tau )}\sum_{i=1}^{r}\sum_{\mathbf{k}(i)+\mathbf{m}\in \Lambda_{\mathbf{k}(i)}} c_{\sigma ,
\mathbf{k}(i) +\mathbf{m}} x_{\sigma}^{A_{\sigma}^{-1}(\beta
-A_{\overline{\sigma}}(\mathbf{k}(i)+\mathbf{m}))}x_{\overline{\sigma}}^{\mathbf{k}(i)+\mathbf{m}}.$$
Let us denote by $v_{\sigma ,\mathbf{k}(i)+\mathbf{m}}$ the exponent
of the monomial $x_{\sigma}^{A_{\sigma}^{-1}(\beta
-A_{\overline{\sigma}}(\mathbf{k}(i)+\mathbf{m}))}x_{\overline{\sigma}}^{\mathbf{k}(i)+\mathbf{m}}$.

\vspace{.3cm}

Since Euler operators $E_i -\beta_i$ annihilate every monomial
$x^{v_{\sigma ,\mathbf{k}(i)+\mathbf{m}}}$ appearing in $f$ we just
need to use toric operators $\Box_u
=\partial^{u_{+}}-\partial^{u_{-}}$ with $u\in L_A =\ker (A)\cap
\Z^{n}$ in order prove that $f$ is annihilated by $H_A (\beta )$ if
and only if the formal series

$$\sum_{\mathbf{k}(i)+\mathbf{m}\in \Lambda_{\mathbf{k}(i)}}
c_{\sigma , \mathbf{k}(i) +\mathbf{m}}
x_{\sigma}^{A_{\sigma}^{-1}(\beta
-A_{\overline{\sigma}}(\mathbf{k}(i)+\mathbf{m}))}x_{\overline{\sigma}}^{\mathbf{k}(i)+\mathbf{m}}$$
is annihilated by $H_A (\beta )$ for all $\sigma \in \mathrm{T}(\tau
)$ and $i=1,\ldots ,r$.

\vspace{.3cm}

This is clear because $v_{\sigma
,\mathbf{k}(i)+\mathbf{m}}-v_{\sigma ' ,\mathbf{k}(j)+\mathbf{m}}\in
\Z^{n}$ if and only if $\sigma =\sigma '$ and $i=j$ (because $\beta$
is very generic and for fixed $\sigma$ we have Lemma
\ref{equiv-lattice}). Recall here that for $u\in L_A$ any pair of
monomials $x^{v}$, $x^{v'}$ verify that
$\partial^{u_{-}}(x^{v})=[v]_{u_{-}} x^{v-u_{-}}$ and
$\partial^{u_{+}}(x^{v'})=[v']_{u_{+}} x^{v'-u_{+}}$ and $
x^{v-u_{-}}=x^{v'-u_{+}}$ if and only if $v-v'=u$.

\vspace{.3cm}

Moreover, a series $\sum_{\mathbf{k}(i)+\mathbf{m}\in
\Lambda_{\mathbf{k}(i)}} c_{\sigma , \mathbf{k}(i) +\mathbf{m}}
x_{\sigma}^{A_{\sigma}^{-1}(\beta
-A_{\overline{\sigma}}(\mathbf{k}(i)+\mathbf{m}))}x_{\overline{\sigma}}^{\mathbf{k}(i)+\mathbf{m}}$
is annihilated by $I_A$ if and only if it is $c
\phi_{\sigma}^{\mathbf{k}(i)}$ for certain $c\in \C$.

\vspace{.3cm}

Thus we obtain that any formal solution of $\M_A (\beta )$ along
$Y_{\tau}$ at $p\in W_{\mathrm{T}(\tau )}\subseteq Y_{\tau}$ is a
linear combination of the linearly independent formal solutions
$\phi_{\sigma}^{\mathbf{k}}$ with $\sigma \in T(\tau )$ and
$\{\Lambda_{k(i)}: \; 1\leq i \leq \operatorname{vol}_{\Z
A}(\Delta_{\sigma})=[\Z A : \Z \sigma]\}$ the partition of
$\N^{n-d}$ associated with $\sigma$ (see Remark \ref{partition}).
That is we have a basis with cardinal:

$$\sum_{\sigma \in T(\tau )}\operatorname{vol}_{\Z
A}(\Delta_{\sigma})=\sum_{\sigma \in T(\tau )}\operatorname{vol}_{\Z
\tau}(\Delta_{\sigma}) [\Z A : \Z \tau ]=\operatorname{vol}_{\Z \tau
}(\Delta_{\tau})[\Z A : \Z \tau ]=\operatorname{vol}_{\Z
A}(\Delta_{\tau}).$$ This finishes the proof of Theorem
\ref{theorem-dimension-formal}.

\section{Irregularity of $\M_A (\beta )$ along coordinate
hyperplanes under some conditions on
$(A,\beta)$}\label{section-pointed}

Assume throughout this section that $A$ is a pointed matrix such
that $\Z A=\Z^d$ and that $Y$ is a coordinate hyperplane. Then we
have that the irregularity complex of order $s$,
$\operatorname{Irr}^{(s)}_Y (\M_{A}(\beta ))$, is a perverse sheaf
on $Y$ (see \cite{Mebkhout3}). This implies in particular the
existence of an analytic subvariety $S\subseteq Y$ with codimension
$q>0$ in $Y$ such that for all $p\in Y\setminus S$:

\begin{equation}
\chi (\operatorname{Irr}_Y^{(s)}(\M_A (\beta ))))_p = \dim
(\mathcal{H}^{0} (\operatorname{Irr}_Y^{(s)}(\M_A (\beta )))_p )
\label{Euler-Poincare}
\end{equation} Here $\chi (\mathcal{F})=\sum_{i\geq 0}(-1)^i \dim
(\mathcal{H}^i(\mathcal{F}))$ denotes the Euler-Poincaré
characteristic of a bounded constructible complex of sheaves
$\mathcal{F}\in \operatorname{D}_{c}^b (\C_Y )$. The characteristic
cycle of $\mathcal{F}\in \operatorname{D}_{c}^b (\C_Y )$ is the
unique lagrangian cycle $$\operatorname{CCh}(\mathcal{F})=m_Y
T_Y^{\ast} Y +\sum_{\alpha :\: \dim Y_{\alpha} < \dim Y} m_{\alpha }
T_{Y_{\alpha}}^{\ast} Y \subseteq T^{\ast} Y$$ that satisfies the
index formula:

$$\chi (\mathcal{F})=  \operatorname{Eu}(m_Y  Y +\sum_{\alpha :\: \dim Y_{\alpha} < \dim
Y}(-1)^{\operatorname{codim}_Y (Y_{\alpha})} m_{\alpha }
\overline{Y_{\alpha}} )$$ where $ \operatorname{Eu}$ denotes the
Euler's morphism between the group of cycles on $Y$ and  the group
of constructible functions on $Y$ with integer values. Thus by
(\ref{Euler-Poincare}) we have that for all $p\in Y\setminus S$:

\begin{equation}
\dim (\mathcal{H}^{0} (\operatorname{Irr}_Y^{(s)}(\M_A (\beta )))_p
= \operatorname{Eu}(
\operatorname{CCh}(\operatorname{Irr}_Y^{(s)}(\M_A (\beta ))) )_p =
m_Y \label{multiplicity-TY}
\end{equation} where $m_Y$ is the multiplicity of $T^{\ast}_Y Y$ in
$\operatorname{CCh}(\operatorname{Irr}_Y^{(s)}(\M_A (\beta )))$.

\vspace{.3cm}

Y. Laurent and Z. Mebkhout provided a formula in
\cite{Laurent-Mebkhout} to obtain the cycle
$\operatorname{CCh}(\operatorname{Irr}_Y^{(s)}(\M_A (\beta )))$ in
terms of the $(1+ \epsilon)$-characteristic cycle and the $(s +
\epsilon )$-characteristic cycle of $\M_A (\beta )$ for $\epsilon
>0$ small enough. Because of the correspondence established in
\cite{Laurent-Mebkhout} in order to compute the multiplicity $m_Y$
of $T^{\ast}_Y Y$ in
$\operatorname{CCh}(\operatorname{Irr}_Y^{(s)}(\M_A (\beta )))$ we
only need to know the multiplicity of $T_{X}^{\ast} X$ and
$T_Y^{\ast} X$ in the $(1+\epsilon ) $-characteristic cycle of $\M_A
(\beta )$ and the $(s + \epsilon )$-characteristic cycle of $\M_A
(\beta )$ with respect to $Y$ for $\epsilon >0$ small enough.

\vspace{.3cm}

We are going to use the multiplicities formula for the
$s$-characteristic cycle of $\M_A (\beta )$ obtained by M. Schulze
and U. Walther in \cite{SW} in the case when $A$ is pointed and
$\beta$ is non-rank-jumping. First of all we need to recall some
definitions given in \cite{SW}.

\vspace{.3cm}

Let us consider $\Phi_A^{s}\ni \tau \subseteq \tau '  \in \Phi_A^{s,
d-1}$ and the natural projection $$\pi_{\tau , \tau '}: \Z \tau '
\rightarrow \Z \tau ' / (\Z \tau ' \cap \Q  \tau ).$$

\begin{defi}\label{def-normal-vol}
In  a lattice $\Lambda $, the volume function
$\operatorname{vol}_{\Lambda }$ is normalized so that the unit
simplex of $\Lambda$ has volume $1$. We abbreviate
$\operatorname{vol}_{\tau ,\tau '}:=\operatorname{vol}_{\pi_{\tau ,
\tau '} (\Z \tau ')}$.
\end{defi}

\begin{defi}
For $\Phi_A^{s}\ni \tau \subseteq \tau '  \in \Phi_A^{s, d-1}$,
define the polyhedra $$P_{\tau , \tau '} :=
\operatorname{conv}(\pi_{\tau , \tau '} (\tau' \cup \{0\} )), \; \;
Q_{\tau , \tau '}:= \operatorname{conv}(\pi_{\tau , \tau '} (\tau'
\setminus \tau ))$$ where $\operatorname{conv}$ means to take the
convex hull.

\end{defi}

The following theorem was proven by M. Schulze and U. Walther (see
\cite[Th. 4.21]{SW} and \cite[Cor. 4.12]{SW}).

\begin{teo}
For  generic $\beta \in \C^d$ (more precisely, non-rank-jumping) and
$\tau \in \Phi_{A}^{s}$, the multiplicity of $\overline{C}_A^{\tau}$
in the $s$-characteristic cycle of $\M_A (\beta )$ is:

$$\mu_A^{s,\tau }=\sum_{\tau \subseteq \tau '   \in \Phi_A^s } [\Z^d
:\Z \tau ] \cdot [(\Z \tau ' \cap \Q \tau  ): \Z \tau ] \cdot
\operatorname{vol}_{\tau ,\tau '}(P_{\tau , \tau '}\setminus Q_{\tau
,\tau '}).$$ Here $\overline{C}_A^{\tau}$ is the closure in
$T^{\ast} X$ of the conormal space to the orbit $ O_A^{\tau
}\subseteq T_0^{\ast} X$, where $ O_A^{\tau }$ is the orbit of
$1_{\tau} \in \{0,1 \}^n$ ($(1_{\tau })_i  =1$ if $a_i \in \tau$,
$(1_{\tau })_i =0$ if $a_i \notin \tau $) by the $d$-torus action:

$$\begin{array}{rcl}
   (\C^{\ast})^d \times T_{0}^{\ast} X & \longrightarrow & T_{0}^{\ast} X \\
    (t , \xi )  & \mapsto & t\cdot \xi := (t^{a_1 } \xi_1 ,\ldots
    ,t^{a_n } \xi_n )
  \end{array}$$
\end{teo}

Assume that $Y=\{x_n = 0\}$ by reordering the variables. We are
interested in the multiplicities of $\overline{C}_A^{ \emptyset }
=T_X^{\ast} X $ and $\overline{C}_A^{ \{ n \} } =T_Y^{\ast} X $ in
the $r$-characteristic cycles of $\M_A (\beta)$ for $r=s+ \epsilon$
and $r=1+ \epsilon$ with $\epsilon >0$ small enough. In particular,
we need to compute $\mu_A^{s +\epsilon, \emptyset}, \mu_A^{s +
\epsilon ,\{n\}}$, $\mu_A^{1+\epsilon , \emptyset }$ and $\mu_A^{1 +
\epsilon , \{n\} }$.

\vspace{.3cm}

It is a classical result that $\mu_A^{1, \emptyset } =
\operatorname{rank}(\M_A (\beta )) = \operatorname{vol}_{\Z^d }
(\Delta_A )$ for generic $\beta$ (see \cite{GKZ}, \cite{Adolphson}).

\vspace{.3cm}

From \cite[Corollay 4.22]{SW} if $\tau =\emptyset $ then $$\mu_A^{s,
\emptyset } = \operatorname{vol}_{\Z^d } (\cup_{\tau ' \in
\Phi_A^{s, d-1}}(\Delta_{\tau '}^1 \setminus
\operatorname{conv}(\tau '))).$$

Since $\Phi_A^{s+\epsilon} $ is constant for $\epsilon >0$ small
enough we have that all its faces $\tau$ are $F$-homogeneous
and then $\operatorname{vol}_{\Z^d }(\operatorname{conv}(\tau ))=0$. As a
consequence,

$$\mu_A^{s+\epsilon, \emptyset}=\operatorname{vol}_{\Z^d }(\cup_{\tau ' \in
\Phi_A^{s +\epsilon  , d-1}}(\Delta_{\tau '}^1 )) =
\operatorname{vol}_{\Z^d }(\cup_{\tau ' \in \Phi_A^{s ,
d-1}}(\Delta_{\tau '}^1 ) ).$$ Let us compute $\mu_A^{r, \{n \}}$
for $r=s+\epsilon$ and $r=1+ \epsilon$.

\vspace{.3cm}

Consider any $\tau \in \Phi_A^{s+\epsilon, d-1}$ such that $n\in
\tau$. Since $\epsilon >0$ is generic ($\Phi_A^{t, d-1}$ is locally
constant at $t=s+\epsilon$) we have that $a_n \notin \Q (\tau
\setminus \{a_n \})$ and hence there exists certain
$(d-1)$-simplices $\sigma_1 ,\ldots , \sigma_r $  such that $n\in
\sigma_i \subseteq \tau$, $\tau =\cup_{i} \sigma_i $, $\sigma_i \cap
\sigma_j$ is an $k$-simplex with $k\leq d-2$ ($\sigma_1 ,\ldots ,
\sigma_r $ is a triangulation of $\tau$). Then
$\operatorname{vol}_{\Z^d}(\Delta_{\tau})= \sum_{i=1}^r
\operatorname{vol}_{\Z^d}(\Delta_{\sigma_i}) $ and we want to prove
that

\begin{equation} \operatorname{vol}_{\Z^d}(\Delta_{\tau})= [\Z^d :
\Z \tau ] \cdot [\Z \tau \cap \Q a_n : \Z a_n ] \cdot
\operatorname{vol}_{\{n\},\tau}(P_{\{n\},\tau } \setminus
Q_{\{n\},\tau }) \label{formula-volume}\end{equation} Since $\Z
\sigma_i \subseteq \Z \tau \subseteq \Z^d$ then
$\operatorname{vol}_{\Z^d}(\Delta_{ \sigma_i})=[\Z^d :\Z \sigma_i ]=
[\Z^d : \Z\tau] \cdot [\Z \tau: \Z \sigma_i ]$ so we only need to
prove:

$$\sum_{i=1}^r [\Z \tau : \Z \sigma_i ]= [\Z \tau \cap \Q a_n : \Z a_n ] \cdot
\operatorname{vol}_{\{n\},\tau}(P_{\{n\},\tau } \setminus
Q_{\{n\},\tau }) .$$ But $a_n \notin \Q (\tau \setminus \{a_n \})$
implies that $[\Z \tau \cap \Q a_n : \Z a_n] =1 $ and $\tau$ is
$F$-homogeneous so we have to prove that:

$$ \operatorname{vol}_{\{n\},\tau}(P_{\{n\},\tau })
=\sum_{i=1}^r [\Z \tau : \Z \sigma_i ].$$ We observe that
$\pi_{\{n\},\tau}(\tau \cup \{0\})=(\tau \setminus \{n\})\cup \{0\}$
in $\Z \tau / (\Z \tau \cap \Q a_n )= \Z (\tau \setminus \{n\})$.
Consider a ($d-2$)-simplex $\widetilde{\sigma }$ such that $\Z
\widetilde{\sigma } =\Z (\tau \setminus \{n\})$. Since $a_n \notin
\sum_{i\in \tau \setminus \{n \}} \Q a_i $ there exists a hyperplane
$H$ such that $a_i \in H $ for all $ i \in \tau\setminus \{n \}\}$,
$0\in H$ and $\widetilde{\sigma} \subseteq H$. Recall that the
Euclidean volume of the convex hull of a bounded polytope $\Delta $
contained in a hyperplane $H \subseteq \R^d$ and a point $c\notin H$
is the product of the relative volume of the polytope
$\operatorname{vol}_{rel}(\Delta )$ and the distance from $c$ to
$H$, $\operatorname{d}(c,H)$, divided by $d!$. Hence, we have the
following equalities:

\vspace{.3cm}

\noindent $\displaystyle
\operatorname{vol}_{\{n\},\tau}(P_{\{n\},\tau })=
\frac{\operatorname{vol}_{rel}(\Delta_{\tau \setminus
\{n\}})}{\operatorname{vol}_{rel}(\Delta_{\widetilde{\sigma}})} =
\frac{\operatorname{vol}
(\Delta_{\tau})}{\operatorname{vol}(\Delta_{\widetilde{\sigma} \cup
\{n \}})} =\sum_{i=1}^r \frac{\operatorname{vol}(\Delta_{\sigma_i
})}{\operatorname{vol}(\Delta_{\widetilde{\sigma} \cup \{n \}})}= \\
= \sum_{i=1}^r \frac{[\Z^d : \Z \sigma_i ]}{[\Z^d : \Z \tau]}=
\sum_{i=1}^r [\Z \tau : \Z \sigma_i ]$.

\vspace{.3cm}

We have proved (\ref{formula-volume}) and as a consequence the
following Lemma.

\begin{lema}\label{lema-multiplicidad}
Consider $s\geq 1$ and $\beta$ non-rank-jumping. Then for all
$\epsilon >0$ small enough:

$$\mu_A^{s+\epsilon ,\{n \}}=\sum_{n \in \tau \in \Phi_A^{s+\epsilon} } \operatorname{vol}_{\Z^d }(\Delta_{\tau }) .$$

\end{lema}

We close this section with the following result about the
irregularity along any coordinate hyperplane $Y$ of the
hypergeometric system $\M_A (\beta )$ associated with a full rank
pointed matrix $A$ with $\Z A=\Z^d$. It is a consequence of Lemma
\ref{lema-multiplicidad} and the results in \cite{Laurent-Mebkhout}.

\begin{teo}\label{Teorema-irreg}
If $\beta\in \C^d$ is generic (more precisely, non-rank-jumping)
then $$\operatorname{dim}_{\C } (\mathcal{H}^0
(\operatorname{Irr}_Y^{(s)} (\M_A (\beta )))_p ) =\sum_{n\notin \tau
\in \Phi_A^{s}}\operatorname{vol}_{\Z^d }(\Delta_{\tau }) -
\sum_{n\notin \tau \in \Phi_A^{1}} \operatorname{vol}_{\Z^d
}(\Delta_{\tau })$$ for all $p\in Y\setminus S$, where $S$ is a
subvariety of $Y$ with $\dim S < \dim Y$. Then, for very generic
$\beta$ the nonzero classes in $\cQ_Y (s)$ of the constructed series
$\phi_{\sigma }^{\mathbf{k}}$ with $\sigma \in \mathrm{T} '$ form a
basis in their common domain of definition $U\subseteq Y$.
\end{teo}

\begin{nota}
Notice that Theorem \ref{Teorema-irreg} implies that under the assumptions of this section equality holds in (\ref{lower-bound-irreg-s}).
\end{nota}

\end{document}